%% file: game.tex
\documentclass[letterpaper,11pt,twoside,keywordsasfootnote,addressatend,noinfoline]{article}

\usepackage{fullpage}
\usepackage[english]{babel}
\usepackage{amssymb}
\usepackage{amsmath}
\usepackage{theorem}
\usepackage{epsfig}
\usepackage{subfigure}
\usepackage{imsart}

\theorembodyfont{\slshape}
\newtheorem{theor}{Theorem}

\newtheorem{lemma}[theor]{Lemma}

\newtheorem{defin}[theor]{Definition}

\newenvironment{proof}{\noindent{\scshape Proof.}}{\hspace{2mm} $\square$}

\newcommand{\Z}{\mathbb{Z}}
\newcommand{\R}{\mathbb{R}}
\newcommand{\D}{\mathbb{D}}
\newcommand{\ind}{\mathbf 1}
\newcommand{\ep}{\epsilon}

\DeclareMathOperator{\card}{card}
\DeclareMathOperator{\sign}{sign}
\DeclareMathOperator{\dist}{dist}

\begin{document}

\begin{frontmatter}

\title     {Stochastic spatial model of producer-consumer \\ systems on the lattice}
\runtitle  {Stochastic spatial model of producer-consumer systems on the lattice}
\author    {N. Lanchier\thanks{Research supported in part by NSF Grant DMS-10-05282.}}
\runauthor {N. Lanchier}
\address   {School of Mathematical and Statistical Sciences, \\ Arizona State University, \\ Tempe, AZ 85287, USA.}

\begin{abstract} \ \
 The objective of this paper is to give a rigorous analysis of a stochastic spatial model of producer-consumer systems that
 has been recently introduced by Kang and the author to understand the role of space in ecological communities in which
 individuals compete for resources.
 Each point of the square lattice is occupied by an individual which is characterized by one of two possible types, and
 updates its type in continuous time at rate one.
 Each individual being thought of as a producer and consumer of resources, the new type at each update is chosen at random
 from a certain interaction neighborhood according to probabilities proportional to the ability of the neighbors to consume
 the resource produced by the individual to be updated.
 In addition to giving a complete qualitative picture of the phase diagram of the spatial model, our results indicate
 that the nonspatial deterministic mean-field approximation of the stochastic process fails to describe the behavior of the
 system in the presence of local interactions.
 In particular, we prove that, in the parameter region where the nonspatial model displays bistability, there is a dominant
 type that wins regardless of its initial density in the spatial model, and that the inclusion of space also translates into
 a significant reduction of the parameter region where both types coexist.
\end{abstract}



\end{frontmatter}


\section{Introduction}
\label{sec:intro}

\indent To understand the role of space in ecological communities in which individuals compete for resources, Kang and the
 author \cite{kang_lanchier_2012} recently introduced a stochastic spatial model of producer-consumer systems on the lattice.
 Based on their numerical simulations of this spatial model and their analytical results for its nonspatial deterministic
 mean-field approximation, they concluded that the inclusion of space drastically affects the outcome of such biological
 interactions.
 Their results for the models with two species are reminiscent of the ones found in \cite{neuhauser_pacala_1999} for the
 spatially explicit Lotka-Volterra model and in \cite{lanchier_neuhauser_2009} for a non Mendelian diploid model:
 in the parameter region where the nonspatial model displays bistability, there is a dominant type that wins regardless of
 its initial density in the spatial model, while the inclusion of space also translates into a significant reduction of
 the parameter region where both species coexist.
 In the presence of three species, other disagreements between the spatial and nonspatial models appear.
 The main purpose of this paper is to give rigorous proofs of some of the results stated in \cite{kang_lanchier_2012}
 for the stochastic spatial model with two species, which we now describe in detail.


\subsection*{Model description.}

\indent Each point of the $d$-dimensional square lattice is occupied by exactly one individual which is characterized by one of two
 possible types, say type 1 and type 2.
 Each individual is thought of as a producer and consumer of resources, and we described the dynamics based on four parameters
 $a_{ij} \geq 0$ that denote the ability of an individual of type $i$ to consume the resource produced by an individual of type $j$,
 that we call resource $j$.
 To avoid degenerate cases, we also assume that each resource can be consumed by at least one type: $a_{1j} + a_{2j} > 0$.
 Each individual dies at rate one, i.e., the type at each vertex is updated at rate one, with the new type being chosen at random
 from a certain neighborhood according to probabilities proportional to the ability of the neighbors to consume the resource
 produced by the individual to be updated.
 More precisely, the state space consists of all functions $\eta$ that map the square lattice into the set of types $\{1, 2 \}$ and
 the dynamics are described by the Markov generator
\begin{equation}
\label{eq:generator-1}
  Lf (\eta) \ = \ \sum_{x \in \Z^d} \ \sum_{i \neq j} \ \frac{a_{ij} \,f_i (x, \eta)}{a_{1j} \,f_1 (x, \eta) + a_{2j} \,f_2 (x, \eta)} \ \ [f (\eta_{x, i}) - f (\eta)]
\end{equation}
 where $f_i (x, \eta)$ denotes the number of type $i$ neighbors
 $$ f_i (x, \eta) \ = \ \card \,\{y \neq x : |y_1 - x_1| + \cdots + |y_d - x_d| \leq M \ \hbox{and} \ \eta (y) = i \} $$
 and where the configuration $\eta_{x, i}$ is obtained from the configuration $\eta$ by assigning type $i$ to vertex
 $x$ and leaving the type of all the other vertices unchanged.
 To avoid cumbersome notations, we will write later $x \sim y$ to indicate that vertices $x$ and $y$ are neighbors, i.e.,
 $$ y \,\neq \,x \quad \hbox{and} \quad |y_1 - x_1| + |y_2 - x_2| + \cdots + |y_d - x_d| \leq M. $$
 The positive integer $M$ is referred to as the range of the interactions.
 The fraction in \eqref{eq:generator-1} means that the conditional probability that the new type is chosen to be type $i$ given
 that vertex $x$ is of type $j$ at the time of an update is equal to the overall ability of the neighbors of type $i$ to consume
 resource $j$ divided by the overall ability of all the neighbors to consume resource $j$.
 Note that the assumption on the parameters, $a_{1j} + a_{2j} > 0$ for $j = 1, 2$, allows to define
 $$ a_1 \ := \ a_{11} \ (a_{11} + a_{21})^{-1} \in \,[0, 1] \qquad \hbox{and} \qquad a_2 \ := \ a_{22} \ (a_{12} + a_{22})^{-1} \in \,[0, 1] $$
 and to rewrite the Markov generator \eqref{eq:generator-1} in the form
\begin{equation}
\label{eq:generator-2}
  Lf (\eta) \ = \ \sum_{x \in \Z^d} \ \sum_{i \neq j} \ \frac{(1 - a_j) \,f_i (x, \eta)}{a_j \,f_j (x, \eta) + (1 - a_j) \,f_i (x, \eta)} \ \ [f (\eta_{x, i}) - f (\eta)]
\end{equation}
 indicating that our model reduces to a two-parameter system.
 Note also that if one of the original parameters in equation \eqref{eq:generator-1} is equal to zero then the denominator in
\begin{equation}
\label{eq:proba}
 \frac{a_{ij} \,f_i (x, \eta)}{a_{1j} \,f_1 (x, \eta) + a_{2j} \,f_2 (x, \eta)} \ = \
 \frac{(1 - a_j) \,f_i (x, \eta)}{a_j \,f_j (x, \eta) + (1 - a_j) \,f_i (x, \eta)} \qquad i \neq j
\end{equation}
 may be equal to zero, in which case the numerator is equal to zero as well.
 When this happens, we have the following alternative and define \eqref{eq:proba} as follows:
\begin{enumerate}
 \item $a_j = f_i = 0$. In this case, motivated by the fact that there is no neighbor of type $i$, we define the probability \eqref{eq:proba}
  at which vertex $x$ becomes type $i$ to be zero. \vspace{4pt}
 \item $a_j = f_i = 1$. In this case, motivated by the fact that there are only neighbors of type $i$, we define the probability \eqref{eq:proba}
  at which vertex $x$ becomes type $i$ to be one.
\end{enumerate}
 We point out that, in addition to be the most natural from a biological point of view, the two assumptions above also make the
 transition probability \eqref{eq:proba} continuous with respect to the parameters, which is a key to some of our proofs.


\subsection*{The mean-field approximation.}

\indent As pointed out in \cite{kang_lanchier_2012}, one of the most interesting aspects of the spatial model is that it results in
 predictions that differ significantly from its nonspatial deterministic mean-field approximation, thus indicating that the
 spatial component plays a key role in these interactions.
 We recall that the mean-field model is obtained by assuming that the population is well-mixing, and refer to Durrett and
 Levin \cite{durrett_levin_1994} for more details.
 This assumption results in a system of coupled differential equations for the densities of each type, and since
 the densities sum up to one, the mean-field model of \eqref{eq:generator-2} reduces to the one-dimensional system
\begin{equation}
\label{eq:mean-field}
 \begin{array}{rcl}
  \displaystyle \frac{du_1}{dt} & = &
  \displaystyle \frac{a_{12} \,u_1 \,u_2}{a_{12} \,u_1 + a_{22} \,u_2} \ - \
                \frac{a_{21} \,u_2 \,u_1}{a_{11} \,u_1 + a_{21} \,u_2} \vspace{12pt} \\ & = &
  \displaystyle \frac{(1 - a_2) \,u_1 \,(1 - u_1)}{a_2 \,(1 - u_1) + (1 - a_2) \,u_1} \ - \ \frac{(1 - a_1) \,u_1 \,(1 - u_1)}{a_1 \,u_1 + (1 - a_1) \,(1 - u_1)} \end{array}
\end{equation}
 where $u_i$ denotes the density of type $i$ individuals.
 Interestingly, the limiting behavior of the mean-field model depends on whether the parameters $a_1$ and $a_2$ are smaller or larger
 than one half, which has a biological interpretation in terms of altruism and selfishness.
 More precisely,
\begin{enumerate}
 \item type $i$ is said to be {\bf altruistic} whenever $a_i < 1/2$ since in this case the resources it produces are more beneficial
  for the other type than for its own type, \vspace{4pt}
 \item type $i$ is said to be {\bf selfish} whenever $a_i > 1/2$ since in this case the resources it produces are more beneficial
  for its own type than for the other type.
\end{enumerate}
 The mean-field model has two trivial equilibria given by $e_1 = 1$ and $e_2 = 0$, respectively.
 By setting the right-hand side of equation \eqref{eq:mean-field} equal to zero and assuming, in order to study the existence of
 nontrivial interior fixed points, that the product $u_1 \,(1 - u_1) \neq 0$, we find
 $$ (1 - a_2) (a_1 - 1/2) \,u_1 \ = \ (1 - a_1) (a_2 - 1/2) \,(1 - u_1) $$
 from which it follows that
 $$ e_* \ = \ \frac{(1 - a_1) (a_2 - 1/2)}{(1 - a_2) (a_1 - 1/2) + (1 - a_1) (a_2 - 1/2)} $$
 is the unique fixed point that may differ from zero and one.
 To analyze the global stability of the fixed points, we observe that the sign of the derivative \eqref{eq:mean-field} is given by
\begin{equation}
\label{eq:sign}
  \begin{array}{l}
  \sign ((1 - a_2) (a_1 - 1/2) \,u_1 - (1 - a_1) (a_2 - 1/2) \,(1 - u_1)) \vspace{4pt} \\ \hspace{20pt} = \
  \sign ((1 - a_2) (a_1 - 1/2) \,(u_1 - e_*) + (1 - a_1) (a_2 - 1/2) \,(u_1 - e_*)), \end{array}
\end{equation}
 which leads to the following three possible regimes for the mean-field model. \vspace{6pt} \\
 1 -- A selfish type always wins against an altruistic type. \vspace{4pt} \\
\begin{proof}
 Without loss of generality, we may assume that $a_1 < 1/2 < a_2$ indicating that type 1 is altruistic whereas type 2 is selfish.
 In this case, the first line of equation \eqref{eq:sign} implies that the derivative of $u_1$ is always negative when $u_1 \in (0, 1)$
 hence there is no fixed point in $(0, 1)$, and the trivial equilibrium $e_1$ is unstable while the trivial equilibrium $e_2$
 is globally stable.
\end{proof} \vspace{6pt} \\
 2 -- In the presence of two selfish types, the system is bistable. \vspace{4pt} \\
\begin{proof}
 Since $a_1 > 1/2$ and $a_2 > 1/2$ we have $e_* \in (0, 1)$.
 The second line of equation \eqref{eq:sign} implies that the derivative of $u_1$ has the same sign as $u_1 - e_*$ so the
 interior fixed point is unstable and the boundary equilibria locally stable:
 the system converges to $e_1$ if the initial density of type 1 is strictly larger than $e_*$, and to $e_2$ if the initial density
 of type 1 is strictly smaller than $e_*$.
\end{proof} \vspace{6pt} \\
 3 -- In the presence of two altruistic types, coexistence occurs. \vspace{4pt} \\
\begin{proof}
 Since $a_1 < 1/2$ and $a_2 < 1/2$ the fixed point $e_*$ again belongs to $(0, 1)$ as in the presence of two selfish types.
 However, the second line in equation \eqref{eq:sign} now implies that the derivative of $u_1$ has the same sign as $e_* - u_1$ so the
 interior fixed point is globally stable whereas the two boundary equilibria $e_1$ and $e_2$ are unstable. \vspace{4pt}
\end{proof}


\subsection*{The stochastic spatial model.}

\indent In order to state our results for the spatial model, we start by defining some of the possible regimes the spin
 system \eqref{eq:generator-2} can exhibit.
\begin{defin} --
 For the stochastic spatial model \eqref{eq:generator-2}, we say that
\begin{enumerate}
 \item type $i$ {\bf wins} if starting from any configuration with infinitely many type $i$,
  $$ \lim_{t \to \infty} \ P \,(\eta_t (x) = i) \ = \ 1 \quad \hbox{for all} \ x \in \Z^d. $$
 \item the system {\bf clusters} if starting from any translation invariant configuration,
  $$ \lim_{t \to \infty} \ P \,(\eta_t (x) = \eta_t (y)) \ = \ 1 \quad \hbox{for all} \ x, y \in \Z^d. $$
 \item type $i$ {\bf survives} if starting from any configuration with infinitely many type $i$,
  $$ \liminf_{t \to \infty} \ P \,(\eta_t (x) = i) \ > \ 0 \quad \hbox{for all} \ x \in \Z^d. $$
 \item both types {\bf coexist} if they both survive.
\end{enumerate}
\end{defin}
 As pointed out above, the analytical results for the spatial and the nonspatial models differ in many aspects, which reveals the
 importance of the spatial component in the interactions.
 These disagreements are more pronounced in low dimensions and emerge when both types are selfish or both types are altruistic, while
 in the presence of one selfish type and one altruistic type the analytical results for both the spatial and nonspatial models agree.

\indent To begin with, we focus on the one-dimensional nearest neighbor process.
 In this case, one obtains a complete picture of the phase diagram based on a simple analysis of the interfaces of the process which
 almost evolve according to a system of annihilating random walks, whereas the more challenging study in higher dimensions and for
 larger ranges of interactions relies on very different techniques.
 We point out that the one-dimensional nearest neighbor process appears as a degenerate case in which coexistence is not possible
 while increasing the spatial dimension or the range of the interactions may result in coexistence of altruistic types.
\begin{theor} --
\label{th:spatial-1D}
 Assume that $M = d = 1$. Then,
\begin{enumerate}
 \item Except for the case $a_1 = a_2 = 1$, the process clusters. \vspace{4pt}
 \item Assuming in addition that $a_1 < a_2$, type 2 wins.
\end{enumerate}
\end{theor}
 The exclusion of the case $a_1 = a_2 = 1$ is simply due to the fact that, under this assumption, a vertex may flip to a new type
 if and only if all its neighbors already have this type, which implies that the process has infinitely many absorbing states.
 While in the absence of space one type wins regardless of the initial densities only in selfish-altruistic systems,
 Theorem \ref{th:spatial-1D} indicates that the most selfish/least altruistic type always wins in the one-dimensional
 nearest neighbor case.
 In particular, coexistence no longer occurs even in altruistic systems.
 Although this case is pathological, it is also symptomatic of the general behavior of the spatial model for which the existence
 of a dominant type that always wins may occur even when both types are altruistic.
 The next result shows that, in any spatial dimension and for any dispersal range, selfish-altruistic interactions result
 in the same outcome in the spatial and nonspatial mean-field models.
\begin{theor} --
\label{th:invasion}
 Assume that $a_1 < 1/2 < a_2$.
 Then, type 2 wins.
\end{theor}
 Even though its rigorous proof is not straightforward, the result is predictable since, when one type is selfish and the other type
 is altruistic, regardless of the type of the vertex to be updated and the spatial configuration, each of the selfish neighbors is
 individually more likely to be selected to determine the new type than each of the altruistic neighbors.
 In contrast, in the case of selfish-selfish interactions covered by Theorem \ref{th:pure-birth} below, the spatial and nonspatial
 models disagree in all dimensions and for all dispersal ranges.
 Recall that in this case the mean-field model is bistable, indicating that the long-term behavior is determined by both the parameters
 and the initial densities.
 The next result shows on the contrary that, including local interactions, even if both types are selfish, enough asymmetry in the
 parameters implies that the most selfish type wins.
\begin{theor} --
\label{th:pure-birth}
 For all $a_1 < 1$ there is $\rho > 0$ such that type 2 wins when $a_2 > 1 - \rho$.
\end{theor}

\begin{figure}[t!]
  \centering
  \scalebox{0.45}{\input{diagrams.pstex_t}}
  \caption{\upshape Phase diagrams of the nonspatial (left) and spatial (right) models.}
\label{fig:diagrams}
\end{figure}
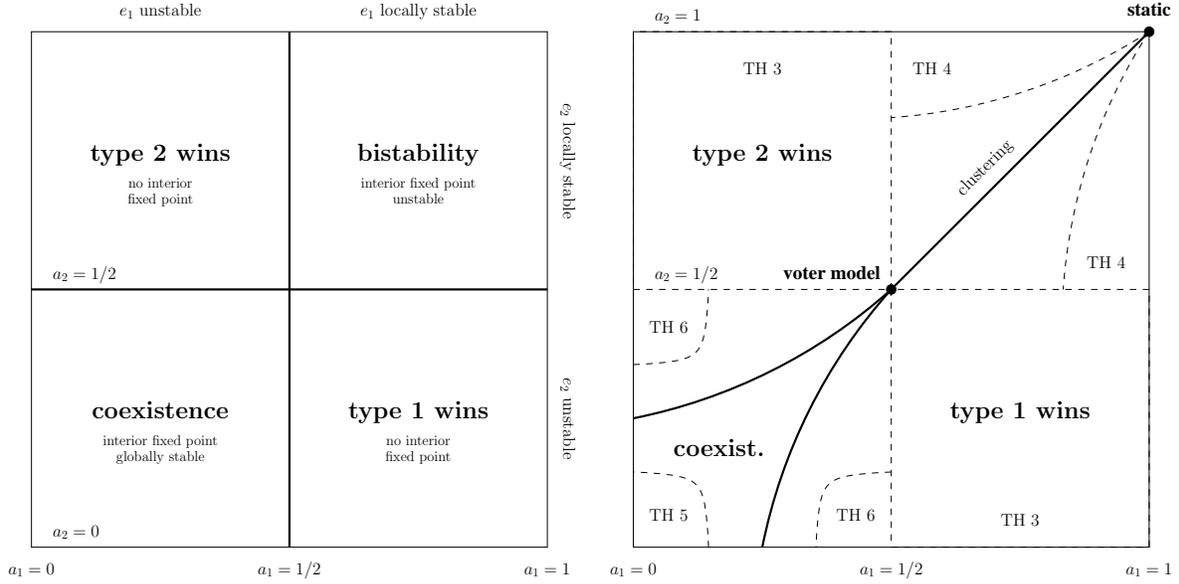

\noindent Numerical simulations in two dimensions indicate more generally that, when type 2 is selfish, it wins whenever $a_1 < a_2$.
 In the neutral case, the obvious symmetry of the evolution rules implies that none of the types wins, but simulations indicate that the
 population again clusters, with boundaries getting sharper as the common value of $a_1 = a_2 > 1/2$ increases.
 This suggests that, in the presence of two selfish types, the long-term behavior of the one-dimensional nearest neighbor process occurs
 in two dimensions as well, a property that we believe is true in any spatial dimension and for any range of interactions.
 Looking finally at altruistic-altruistic interactions, recall that coexistence always occurs in the mean-field approximation whereas the
 least altruistic type always wins in the one-dimensional nearest neighbor process.
 Our last two results show that the long-term behavior of the process in higher spatial dimensions and/or for larger dispersal ranges is
 intermediate between these two extreme behaviors.
\begin{theor} --
\label{th:coexistence}
 For all $(M, d) \neq (1, 1)$, there is $\rho > 0$ such that coexistence occurs when
 $$ a_1 < \rho \quad \hbox{and} \quad a_2 < \rho. $$
\end{theor}
 In words, sufficiently strong mutual altruism translates into coexistence of both types.
 Numerical simulations of the two-dimensional system suggest in addition that two altruistic types coexist in the neutral case
 $a_1 = a_2 < 1/2$ therefore the coexistence region of the spatial model stretches up to the parameter point corresponding to voter model
 interactions.
 Interestingly, our last result shows that there exists a parameter region for which even altruistic types cannot coexist, thus
 indicating that the inclusion of local interactions translates into a reduction of the coexistence region in general and not only in the
 one-dimensional nearest neighbor case.
\begin{theor} --
\label{th:reduction}
 There exists $\rho > 0$ such that type 2 wins whenever
 $$ a_1 < \rho \quad \hbox{and} \quad a_2 > 1/2 - \rho. $$
\end{theor}
 We note that, even though the inclusion of local interactions always translate into a reduction of the coexistence
 region, numerical simulations indicate that this reduction becomes less pronounced while increasing the spatial dimension or
 the range of the interactions.
 The results for the mean-field model and spatial model when $(M, d) \neq (1, 1)$ are summarized in Figure \ref{fig:diagrams}.
 In the phase diagram of the spatial model, the dashed regions are the ones that are covered by our theorems, while the thick
 lines are the critical curves suggested by simulations.
 For a detailed analysis of the nonspatial model and additional results based on numerical simulations for the spatial model
 in the presence of three species, we refer the reader to the companion paper \cite{kang_lanchier_2012}.

\indent The rest of this paper is devoted to the proofs.
 To avoid cumbersome notations and messy calculations, we prove all our results, with the exception of Theorem \ref{th:coexistence},
 when the range of the interactions is equal to one.
 However, we point out that all the techniques used in the proofs easily extend to larger dispersal ranges.
 The reason for excluding Theorem \ref{th:coexistence} from this rule is that it is the only result that shows a disagreement
 between the one-dimensional nearest neighbor case and the other cases, thus indicating that the range of the interactions
 plays a specific role in the presence of strong altruistic-altruistic interactions.




\section{The one-dimensional spatial model}
\label{sec:spatial-1D}

\indent This section is devoted to the proof of Theorem \ref{th:spatial-1D} which gives a complete picture of the long-term behavior
 of the one-dimensional process with nearest neighbor interactions.
 To begin with, we study the evolution in the non-neutral case, and look at the process starting with only type 2 individuals in an
 interval of length $N + 1$.
 The following lemma shows that the survival probability of the individuals of type 1 decreases exponentially with $N$ when $a_1 < a_2$.
\begin{lemma} --
\label{lem:drift}
 Assume that $M = d = 1$ and $a_1 < a_2$. Then,
 $$ P \,(\eta_t (x) \to 2 \ \hbox{for all} \ x \in \Z \ | \ \eta_0 (x) = 2 \ \hbox{for all} \ x \in [0, N] \cap \Z) \ \geq \ 1 - c_0^N $$
 where $c_0 := (1 - a_2) \,(1 - a_1)^{-1} < 1$.
\end{lemma}
\begin{proof}
 First, we observe that the process is attractive, which directly follows from the monotonicity of the probabilities in \eqref{eq:proba}
 with respect to the fractions of neighbors of each type.
 In particular, it suffices to prove the result when starting with exactly $N + 1$ individuals of type 2.
 To understand how the number of type 2 vertices evolve, let $x \in \Z$ and assume that the neighbors of $x$ have different types.
 Then, the rate at which vertex $x$ jumps from type 2 to type 1 is
\begin{equation}
\label{eq:left}
  c_- \ := \ \frac{(1 - a_2) \,f_1 (x, \eta)}{(1 - a_2) \,f_1 (x, \eta) + a_2 \,f_2 (x, \eta)} \ = \
      \frac{1 - a_2}{(1 - a_2) + a_2} \ = \ 1 - a_2
\end{equation}
 while the rate at which it jumps from type 1 to type 2 is
\begin{equation}
\label{eq:right}
  c_+ \ := \ \frac{(1 - a_1) \,f_2 (x, \eta)}{a_1 \,f_1 (x, \eta) + (1 - a_1) \,f_2 (x, \eta)} \ = \
      \frac{1 - a_1}{a_1 + (1 - a_1)} \ = \ 1 - a_1.
\end{equation}
 In particular, as long as there are at least two individuals of type 2, the number of such individuals performs a random walk
 that jumps to the right at rate $2 c_+$ and to the left at rate $2 c_-$ since, due to one-dimensional nearest neighbor interactions
 and the choice of the initial configuration, at all times the set of vertices of type 2 is an interval.
 Moreover, on the event that type 2 survives, the leftmost vertex of type 2 converges almost surely to $- \infty$ while the rightmost
 vertex of type 2 converges almost surely to $+ \infty$, therefore the probability to be estimated is equal to the probability that
 the asymmetric random walk described above tends to infinity.
 The result then reduces to a standard estimate for asymmetric random walks.
 More precisely, let $Z_n$ denote the number of individuals of type 2 after $n$ updates of the system, and let
 $$ P_i \ := \ P \,(Z_n = 1 \ \hbox{for some} \ n \geq 0 \ | \ Z_0 = i). $$
 Then, using a first-step analysis, for all $i \geq 2$ we have
 $$ \begin{array}{rcl}
    (c_+ + c_-) \,P_i & = & c_+ \,P_{i + 1} + c_- \,P_{i - 1} \quad \hbox{and} \vspace{4pt} \\
      P_{i + 1} - P_i & = & c_0 \,(P_i - P_{i - 1}) \ = \ \cdots \ = \ c_0^{i - 1} \,(P_2 - P_1) \end{array} $$
 where $c_0 = c_- / c_+ = (1 - a_2) \,(1 - a_1)^{-1}$.
 Since $P_1 = 1$, we deduce that $P_i = c_0^{i - 1}$ and
 $$ \begin{array}{l}
  P \,(\eta_t (x) \to 2 \ \hbox{for all} \ x \in \Z \ | \ \eta_0 (x) = 2 \ \hbox{for all} \ x \in [0, N] \cap \Z) \vspace{4pt} \\ \hspace{20pt} \geq \
  P \,(\eta_t (x) \to 2 \ \hbox{for all} \ x \in \Z \ | \ \{x : \eta_0 (x) = 2 \} = [0, N] \cap \Z) \ \geq \ 1 - P_{N + 1} \ = \ 1 - c_0^N. \end{array} $$
 This completes the proof.
\end{proof} \\ \\
 With Lemma \ref{lem:drift} in hands, we can now establish the first part of Theorem \ref{th:spatial-1D} in the non-neutral case
 as well as the second part of the theorem.
\begin{lemma} --
\label{lem:cone}
 Assume that $M = d = 1$ and $a_1 < a_2$. Then, type 2 wins.
\end{lemma}
\begin{proof}
 Let $N$ be a positive integer.
 Then, starting from any initial configuration with infinitely many vertices of type 2, with probability one, there exist $z \in \Z$ and
 $t < \infty$ such that
 $$ \eta_t (x) = 2 \ \ \hbox{for} \ \ x = z, z + 1, \ldots, z + N. $$
 Using in addition that the process is attractive, and that the evolution rules are translation invariant in space and time,
 Lemma \ref{lem:drift} implies that
 $$ \begin{array}{l}
  P \,(\eta_t (x) \to 2 \ \hbox{for all} \ x \in \Z \ | \,\card \,\{x : \eta_0 (x) = 2 \} = \infty) \vspace{4pt} \\ \hspace{40pt} \geq \
  P \,(\eta_t (x) \to 2 \ \hbox{for all} \ x \in \Z \ | \ \eta_0 (x) = 2 \ \hbox{for all} \ x \in [0, N] \cap \Z) \ \geq \ 1 - c_0^N. \end{array} $$
 Since this holds for all $N$ and since $c_0 < 1$, the lemma follows.
\end{proof} \\ \\
 To complete the proof of Theorem \ref{th:spatial-1D}, it remains to show that, when $a_1 = a_2 \neq 1$, the process clusters.
 Due to the particular geometry of the configurations in the one-dimensional nearest neighbor case, the proof reduces to the analysis
 of an auxiliary process that we shall call the interface process: the spin system $(\xi_t)$ defined on the translated lattice
 $\D = \Z + 1/2$ by setting
 $$ \xi_t (v) \ = \ \bigg|\,\eta_t \bigg(v + \frac{1}{2} \bigg) - \eta_t \bigg(v - \frac{1}{2} \bigg) \bigg| \quad
    \hbox{for all} \ \ v \in \D. $$
 In words, the process has a one at sites $v \in \D$ which are located between two vertices that have different types, and a zero at sites
 which are located between two vertices that have the same type, so the process keeps track of the interfaces of the spatial model.
\begin{lemma} --
\label{lem:extinction}
 Assume that $M = d = 1$ and $a_1 = a_2 \neq 1$. Then, the process clusters.
\end{lemma}
\begin{proof}
 Thinking of each site $v \in \D$ as being occupied by a particle if $\xi (v) = 1$ and as empty otherwise, the idea of the proof is
 to establish almost sure extinction of this system of particles, which is equivalent to clustering of the spatial model.
 First, the proof of Lemma \ref{lem:drift} indicates that a vertex of either type, say type $i$, changes its type \vspace{-2pt}
\begin{enumerate}
 \item at rate one if none of its two nearest neighbors is of type $i$, \vspace{4pt}
 \item at rate $c_- = c_+ > 0$ defined in \eqref{eq:left} and \eqref{eq:right} if its neighbors have different types.
\end{enumerate}
 This induces the following dynamics for the interface process: a particle at site $v \in \D$ jumps to each of its empty nearest neighbors
 at rate $c_+$ and two particles distance one apart annihilate each other at rate one.
 More formally, the Markov generator is given by
 $$ \begin{array}{rcl} L_{\xi} f (\xi) & = &
    \displaystyle \sum_{v \in \D} \ \sum_{|w - v| = 1} c_+ \ \ind \{\xi (v) \neq \xi (w) \} \ [f (\xi_{v \leftrightarrow w}) - f (\xi)] \\ && \hspace{60pt} + \
    \displaystyle \sum_{v \in \D} \ \sum_{|w - v| = 1} \ind \{\xi (v)  =   \xi (w) \} \ [f (\xi_{v, w}) - f (\xi)] \end{array} $$
 where configuration $\xi_{v \leftrightarrow w}$ is obtained from $\xi$ by exchanging the contents of vertices $v$ and $w$ while
 configuration $\xi_{v, w}$ is obtained by killing the particles at $v$ and $w$ if they exist.
 Since particles can only annihilate, the probability that a given site is occupied by a particle decreases in time
 so it has a limit when time goes to infinity.
 Since in addition the evolution rules of the process are translation invariant, this limit does not depend on the site under
 consideration:
 $$ \hbox{there exists $a$ such that} \ \ \lim_{t \to \infty} \ P \,(\xi_t (v) = 1) = a \ \ \hbox{for all} \ v \in \D. $$
 In other respects, given any two particles alive at time $t$, at some random time $s > t$ almost surely finite, either one of
 these particles is killed due to a collision with a third particle or both particles annihilate each other.
 The latter follows immediately from the fact that one-dimensional random walks are recurrent while each time two particles are
 distance one apart, there is a positive probability that they annihilate at their next jump when $a_1 = a_2 \neq 1$.
 This implies that the limit $a$ must equal zero, from which we conclude that
 $$ \lim_{t \to \infty} \ P \,(\eta_t (x) \neq \eta_t (y)) \ \leq \
    \lim_{t \to \infty} \ \sum_{i = 1}^{y - x} \ P \,\bigg(\xi_t \bigg(x + i - \frac{1}{2} \bigg) = 1 \bigg) \ = \ 0 $$
 for all $x, y \in \Z$.
 This completes the proof.
\end{proof}


\section{Altruistic-selfish interactions}
\label{sec:invasion}

\indent This section is devoted to the proof of Theorem \ref{th:invasion} which states that, in the presence of altruistic-selfish
 interactions, the selfish type always wins.
 The proof is divided into four steps.
 The first step is to show that, under the assumptions of the theorem, the set of type 2 dominates its counterpart in a process that
 we shall call perturbation of the voter model.
 In particular, in order to establish the theorem, it suffices to prove that type 2 wins for this new process.
 The reason for introducing a perturbation of the voter model is that, contrary to process \eqref{eq:generator-2}, it can be studied
 using duality techniques, and the second step of the proof is to describe its dual process in detail, while the third step is to
 construct selected dual paths that are key to proving that type 2 survives.
 Finally, the fourth step combines these selected dual paths to a block construction in order to prove that not only type 2 survives
 but also type 1 goes extinct.
 Before giving the details of the proof, we note that the third step will be used again in the proof of Theorem \ref{th:reduction}
 while the fourth step in the proof of both Theorems \ref{th:pure-birth} and \ref{th:reduction}, but these two steps are detailed
 in this section only.
 We also point out that Theorem \ref{th:invasion} can be proved without the use of a block construction, but since it is needed in
 the proof of Theorems \ref{th:pure-birth} and \ref{th:reduction}, we follow the same approach for all three theorems.


\subsection*{Coupling with a voter model perturbation.}

\indent We first observe that, under the assumptions of the theorem, there exists a constant $\rho > 0$ fixed from now on such that
\begin{equation}
\label{eq:parameters}
  a_1 \ < \ \frac{1 - \rho}{2} \qquad \hbox{and} \qquad a_2 \ > \ \frac{1 + \rho}{2}
\end{equation}
 in which case the set of type 2 for the spatial model \eqref{eq:generator-2} dominates stochastically its counterpart in a certain
 perturbation of the voter model that we denote later by $(\xi_t)$.
 The dynamics of this voter model perturbation depend on a single parameter:
\begin{equation}
\label{eq:epsilon}
  \ep \ := \ \frac{\rho}{d (1 - \rho) + \rho}
\end{equation}
 and can be described as follows.
 As in the original spatial model, the type at each vertex $x$ is updated at rate one, but the new type is now chosen according to the following rules:
\begin{enumerate}
 \item with probability $1 - \ep > 0$, the new type at vertex $x$ is chosen uniformly at random from the set of the nearest neighbors, \vspace{4pt}
 \item with the residual probability $\ep > 0$, the new type at vertex $x$ is chosen to be type 1 if all the nearest neighbors are of type 1,
  and type 2 otherwise.
\end{enumerate}
 More formally, the dynamics are described by the Markov generator
\begin{equation}
\label{eq:voter}
 \begin{array}{rcl} L_{\xi} f (\xi) & = &
   \displaystyle \sum_{x \in \Z^d} \ [(1 - \ep) \,f_1 (x, \xi) + \ep \ \ind \{f_2 (x, \xi) = 0 \}]
                 \ [f (\xi_{x, 1}) - f (\xi)] \vspace{-4pt} \\ && \hspace{25pt} + \
   \displaystyle \sum_{x \in \Z^d} \ [(1 - \ep) \,f_2 (x, \xi) + \ep \ \ind \{f_2 (x, \xi) \neq 0 \}] \ [f (\xi_{x, 2}) - f (\xi)] \end{array}
\end{equation}
 where configuration $\xi_{x, i}$ is obtained from $\xi$ by assigning type $i$ to vertex $x$ and leaving the type of all the
 other vertices unchanged.
 Then, we have the following result.
\begin{lemma} --
\label{lem:coupling}
 Assume that $\eta_0 \equiv \xi_0$ and \eqref{eq:parameters} holds. Then,
 $$ P \,(\eta_t (x) = 2) \,\geq \,P \,(\xi_t (x) = 2) \quad \hbox{for all} \ x \in \Z^d \ \hbox{and all} \ t > 0. $$
\end{lemma}
\begin{proof}
 This follows from certain inequalities between the transition rates.
 When all the neighbors of vertex $x$ are of the same type at the time of an update, then vertex $x$ becomes of this type with
 probability one in both processes.
 Also, given that $x$ is of type 2 and has at least one neighbor of each type, the rate at which it becomes 1 for
 process \eqref{eq:generator-2} is bounded from above by
\begin{equation}
\label{eq:upper-bound}
  \begin{array}{l}
    \displaystyle \frac{(1 - a_2) \,f_1 (x, \eta)}{(1 - a_2) \,f_1 (x, \eta) + a_2 \,f_2 (x, \eta)} \ \leq \
    \displaystyle \frac{(1 - \rho) \,f_1 (x, \eta)}{(1 - \rho) \,f_1 (x, \eta) + (1 + \rho) \,f_2 (x, \eta)} \vspace{10pt} \\ \hspace{60pt} = \
    \displaystyle \frac{(1 - \rho) \,f_1 (x, \eta)}{(1 - \rho) + 2 \rho \,f_2 (x, \eta)} \ \leq \
    \displaystyle \frac{d (1 - \rho)}{d (1 - \rho) + \rho} \ f_1 (x, \eta) \ = \
    \displaystyle (1 - \ep) \,f_1 (x, \eta) \end{array}
\end{equation}
 which is the rate at which it becomes 1 for process \eqref{eq:voter}.
 Similarly, given that $x$ is of type 1 and has at least one neighbor of each type, the rate at which it becomes 2 for
 process \eqref{eq:generator-2} is
\begin{equation}
\label{eq:lower-bound}
  \begin{array}{l}
    \displaystyle \frac{(1 - a_1) \,f_2 (x, \eta)}{a_1 \,f_1 (x, \eta) + (1 - a_1) \,f_2 (x, \eta)} \ \geq \
    \displaystyle \frac{(1 + \rho) \,f_2 (x, \eta)}{(1 - \rho) \,f_1 (x, \eta) + (1 + \rho) \,f_2 (x, \eta)} \vspace{10pt} \\ \hspace{40pt} = \
    \displaystyle \frac{(1 + \rho) \,f_2 (x, \eta)}{(1 - \rho) + 2 \rho \,f_2 (x, \eta)} \ \geq \
    \displaystyle \frac{d (1 - \rho) \,f_2 (x, \eta) + \rho}{d (1 - \rho) + \rho} \ = \ (1 - \ep) \,f_2 (x, \eta) + \ep \end{array}
\end{equation}
 which is the rate at which it becomes 2 for process \eqref{eq:voter}.
 To see that the last inequality is indeed true, we introduce the functions
 $$ g (z) \ = \ \frac{(1 + \rho) \,z}{(1 - \rho) + 2 \rho \,z} \qquad \hbox{and} \qquad
    h (z) \ = \ \frac{d (1 - \rho) \,z + \rho}{d (1 - \rho) + \rho} $$
 and observe that we have the equalities
 $$ g (1) \ = \ h (1) \ = \ 1 \qquad \hbox{and} \qquad g ((2d)^{-1}) \ = \ h ((2d)^{-1}) \ = \ \frac{1 + \rho}{2d (1 - \rho) + 2 \rho}. $$
 Since in addition $g$ is concave on $(0, 1)$ whereas $h$ is linear, the last inequality in \eqref{eq:lower-bound} follows.
 The inequalities on the transition rates in \eqref{eq:upper-bound} and \eqref{eq:lower-bound} give the lemma.
\end{proof}


\subsection*{Duality with branching coalescing random walks.}

\indent In view of Lemma \ref{lem:coupling}, standard coupling arguments -- see Liggett \cite{liggett_1985}, Section II.1, for details
 about coupling techniques -- imply that, to prove the theorem, it suffices to prove that type 2 wins for the process \eqref{eq:voter}.
 Recall that the reason for introducing this process is that, contrary to the original spatial model, it can be studied using duality
 techniques.
 First, we construct the process \eqref{eq:voter} graphically as follows:
\begin{enumerate}
 \item for each oriented edge $e = (x, y) \in \Z^d \times \Z^d$ with $x \sim y$
 \begin{enumerate}
  \item we let $\Lambda (x, y)$ be a Poisson point process with intensity $(2d)^{-1} (1 - \ep)$,
  \item we draw an arrow from vertex $y$ to vertex $x$ at times $t \in \Lambda (x, y)$ to indicate that the individual at $x$
   mimics the individual at $y$.
 \end{enumerate}
 \item for each vertex $x \in \Z^d$
 \begin{enumerate}
  \item we let $\Delta (x)$ be a Poisson point process with intensity $\ep$,
  \item we draw a set of $2d$ arrows from each $y \sim x$ to vertex $x$ at times $t \in \Delta (x)$ to indicate that $x$ becomes
   of type 1 if all its neighbors are of type 1, and of type 2 otherwise.
 \end{enumerate}
\end{enumerate}
 We refer the reader to the left-hand side of Figure \ref{fig:dual} for an illustration of the graphical representation.
 The type at any space-time point can be deduced from the graphical representation and the configuration of the system at earlier
 times by going backwards in time.
 Declare that a dual path from space-time point $(x, T)$ down to $(y, T - s)$ exists, which we write $(x, T) \downarrow (y, T - s)$,
 whenever there are sequences of times and vertices
 $$ s_0 = T - s \ < \ s_1 \ < \ \cdots \ < \ s_{m + 1} = T \qquad \hbox{and} \qquad x_0 = y, \,x_1, \,\ldots, \,x_m = x $$
 such that the following two conditions hold:
\begin{enumerate}
 \item for $i = 1, 2, \ldots, m$, there is an arrow from $x_{i - 1}$ to $x_i$ at time $s_i$ and \vspace{4pt}
 \item for $i = 0, 1, \ldots, m$, there is no arrow that points at the segments $\{x_i \} \times (s_i, s_{i + 1})$.
\end{enumerate}
 The dynamics imply that vertex $x$ is of type 1 at time $T$ if and only if
\begin{equation}
\label{eq:duality}
 \xi_0 (y) = 1 \ \ \hbox{for all} \ \ y \in \Z^d \ \hbox{such that} \ (x, T) \downarrow (y, 0).
\end{equation}
 Note that the set-valued process
  $$ \hat \xi_s (x, T) \ = \ \{y \in \Z^d : (x, T) \downarrow (y, T - s) \} \quad \hbox{for all} \ 0 \leq s \leq T $$
 defines a system of branching coalescing random walks where particles jump at rate $1 - \ep$ to one of their nearest neighbors
 chosen uniformly at random and split at rate $\ep$ into $2d$ particles which are sent to each of their neighbors.
 In addition, whenever two particles are located on the same vertex they coalesce.
 See the left-hand side of Figure \ref{fig:dual} for an illustration where the system of random walks is represented in thick lines.
 This process is reminiscent, though not identical, of the dual process of the biased voter model \cite{bramson_griffeath_1981} and we
 also refer to \cite{athreya_swart_2005, krone_neuhauser_1997} for similar dual processes.


\subsection*{Selected dual paths and random walk estimates.}

\indent To bound the probability that a given space-time point $(x, T)$ is of type 1, we now construct a dual path $(X_t (x))$
 that keeps track of a specific particle in the dual process.
 For each $i = 1, 2, \ldots, d$, we define the hyperplane
 $$ H_i \ := \ \{z = (z_1, z_2, \ldots, z_d) \in \R^d : z_i = 0 \} $$
 as well as the deterministic times $T_i = i \,c_1 N$ where $c_1 = 4 \ep^{-1}$.
 Recall that the constant $\ep$ is defined in equation \eqref{eq:epsilon} above.
 Also, we let $\dist (x, H)$ denote the Euclidean distance between a vertex $x$ and its orthogonal projection on a set $H$.
 The dual path starts at $X_0 (x) = x$ and is defined recursively from the graphical representation as follows.
 For all $t > 0$, define
 $$ s (t) \ = \ \inf \,\{s > t : T - s \in \Delta (X_t (x)) \ \hbox{or} \ T - s \in \Lambda (X_t (x), y) \ \hbox{for some} \ y \sim X_t (x) \}. $$
 For all $s \in (t, s (t))$, we set $X_s (x) = X_t (x)$ while at time $s (t)$ we have the alternative:
\begin{enumerate}
 \item If $T - s (t) \in \Lambda (X_t (x), y)$ for some $y \sim X_t (x)$ then $X_{s (t)} (x) = y$. \vspace{4pt}
 \item If $T - s (t) \in \Delta (X_t (x))$ and $t \in (T_{i - 1}, T_i)$ for some $i = 1, 2, \ldots, d$, then
  $$ X_{s (t)} (x) \sim X_t (x) \quad \hbox{and} \quad \dist (X_{s (t)} (x), H_i) < \dist (X_t (x), H_i). $$
 \item If $T - s (t) \in \Delta (X_t (x))$ and $t > T_d$ then
  $$ P \,(X_{s (t)} (x) = y) \ = \ (2d)^{-1} \quad \hbox{for all} \ y \sim X_t (x). $$
\end{enumerate}
 Note that there is a unique $X_{s (t)} (x)$ that satisfies the second condition above so the dual path is well-defined.
 In words, the dual path travels backwards in time down the graphical representation following the arrows from tip to tail.
 When a set of $2d$ arrows is encountered, the process jumps in the direction that makes it closer to the hyperplane $H_i$ between
 time $T_{i - 1}$ and time $T_i$ while after the last time $T_d$ the process jumps uniformly at random in all directions.
 This, together with the properties of the graphical representation introduced above, implies that the selected dual path
 always jumps at rate one.
 At each jump, with probability $1 - \ep$, the target site is chosen uniformly at random from the set of the nearest neighbors
 whereas, with probability $\ep$, the process moves so as to decrease its distance to $H_1$ until time $T_1$ when
 a similar mechanism operates in the direction of the second axis, and so on.
 In order to later compare the process \eqref{eq:voter} with oriented site percolation, we start by collecting important properties of
 the selected dual path.
 These properties are based on random walk estimates and are given in the following three lemmas.
\begin{lemma} --
\label{lem:go}
 There exist $C_1 < \infty$ and $\gamma_1 > 0$ such that, for all $i = 1, 2, \ldots, d$,
 $$ P \,(\dist (X_{T_i} (x), H_i) > N/2 \ | \,\dist (X_{T_{i - 1}} (x), H_i) \leq 4N) \ \leq \ C_1 \,\exp (- \gamma_1 N) $$
 for all $N$ sufficiently large.
\end{lemma}
\begin{proof}
 First, we observe that for all $t \in (T_{i - 1}, T_i)$ we have
 $$ \dist (X_t (x), H_i) \ \to \ \left\{\hspace{-3pt} \begin{array}{lll}
    \dist (X_t (x), H_i) + 1 & \hbox{at rate} & (2d)^{-1} (1 - \ep) \vspace{4pt} \\
    \dist (X_t (x), H_i) - 1 & \hbox{at rate} & (2d)^{-1} (1 - \ep) + \ep \end{array} \right. $$
 from which we deduce that, for all $t \in (T_{i - 1}, T_i)$,
\begin{equation}
\label{eq:go-1}
  \lim_{h \,\to \,0} \ h^{-1} \,E \,(\dist (X_{t + h} (x), H_i) - \dist (X_t (x), H_i)) \ = \ - \ep.
\end{equation}
 In particular, recalling the definition of $T_i$ and $T_{i - 1}$, we obtain
 $$ E \,(\dist (X_{T_i} (x), H_i) \ | \,\dist (X_{T_{i - 1}} (x), H_i) = 4N) \ = \ 4N - \ep \,(T_i - T_{i - 1}) \ = \ 0. $$
 The lemma then follows from large deviation estimates for the Poisson distribution and the Binomial distribution together with
 standard coupling arguments.
\end{proof}
 \begin{lemma} --
\label{lem:stay}
 For any constant $C > 4d \ep^{-1}$ and all $i = 1, 2, \ldots, d$,
 $$ P \,(\dist (X_t (x), H_i) > N \ \hbox{for some} \ t \in (T_i, C N) \ | \,\dist (X_{T_i} (x), H_i) \leq N/2) \ \leq \ 3 \,\exp (- \sqrt N) $$
 for all $N$ sufficiently large.
\end{lemma}
\begin{proof}
 Denote by $(Z_n)$ the discrete-time one-dimensional simple symmetric random walk and introduce the discrete random variable $J_s$
 that counts the number of times the $i$th coordinate of the process $(X_t (x))$ jumps by time $s$.
 Since the $i$th coordinate evolves according to the continuous-time simple symmetric random walk run at rate $d^{-1} (1 - \ep)$ after
 time $T_i$, we have
 $$ \begin{array}{l}
  P \,(\dist (X_t (x), H_i) > N \ \hbox{for some} \ t \in (T_i, C N) \ | \,\dist (X_{T_i} (x), H_i) \leq N/2) \vspace{4pt} \\ \hspace{25pt} \leq \
  P \,(J_{CN} > N_0) \ + \ P \,(Z_n \notin (-N, N) \ \hbox{for some} \ n \leq N_0 \ | \ Z_0 \in (- N / 2, N / 2)) \end{array} $$
 for all $N_0 > 0$.
 Taking $N_0 = d^{-1} CN$, large deviation estimates for the Poisson distribution imply that the first term on the right-hand side can
 be bounded by
 $$ P \,(J_{CN} > N_0) \ \leq \ P \,(J_{CN} > d^{-1} (1 - \ep) \,CN) \ \leq \ C_2 \,\exp (- \gamma_2 N) \ \leq \ \exp (- \sqrt N) $$
 for suitable constants $C_2 < \infty$ and $\gamma_2 > 0$, and all $N$ sufficiently large.
 To estimate the second term, we use the reflection principle and Chebyshev's inequality.
 For all $\theta > 0$, we have
 $$ \begin{array}{l}
  P \,(Z_n \notin (-N, N) \ \hbox{for some} \ n \leq N_0 \ | \ Z_0 \in (- N / 2, N / 2)) \vspace{4pt} \\ \hspace{20pt} \leq \
  P \,(Z_n \notin (-N / 2, N / 2) \ \hbox{for some} \ n \leq N_0 \ | \ Z_0 = 0) \vspace{4pt} \\ \hspace{20pt} \leq \
  2 \ P \,(Z_{N_0} \notin (-N / 2, N / 2) \ | \ Z_0 = 0) \ \leq \
  2 \,\exp (- \theta N / 2) \ E \,(\exp (\theta Z_{N_0}) \ | \ Z_0 = 0) \vspace{4pt} \\ \hspace{20pt} \leq \
  2 \,\exp (- \theta N / 2) \ [E \,(\exp (\theta Z_1) \ | \ Z_0 = 0)]^{N_0} \ \leq \
  2 \,\exp (- \theta N / 2 + N_0 \ln \phi (\theta)) \end{array} $$
 where the function $\phi (\theta)$ is the moment generating function of $Z_1$. Since
 $$ \ln \phi (\theta) \ = \ \ln \bigg(\frac{e^{\theta} + e^{- \theta}}{2} \bigg) \ = \
    \ln \bigg(1 + \frac{\theta^2}{2} + o (\theta^2) \bigg) \ \leq \ \theta^2 $$
 when $\theta > 0$ is small enough, taking $\theta = 4 / \sqrt N$, we can conclude that
 $$ \begin{array}{l}
  P \,(Z_n \notin (-N, N) \ \hbox{for some} \ n \leq N_0 \ | \ Z_0 \in (- N / 2, N / 2)) \vspace{4pt} \\ \hspace{20pt} \leq \
  2 \,\exp (- \theta N / 2 + N_0 \ln \phi (\theta)) \ \leq \ 2 \,\exp (- 2 \sqrt N + 16 d^{-1} C) \ \leq \ 2 \,\exp (- \sqrt N) \end{array} $$
 for all $N$ sufficiently large.
 This completes the proof.
\end{proof}
\begin{lemma} --
\label{lem:target}
 Let $C_3 > 4 d \ep^{-1}$ and $x \in (- 2N, 2N]^d$.
 Then, for all $N$ large,
 $$ P \,(X_t (x) \notin (- 4N, 4N]^d \ \hbox{for some} \ t < C_3 N \ \hbox{or} \ X_{C_3 N} (x) \notin (- N, N]^d) \ \leq \ 7d \,\exp (- \sqrt N). $$
\end{lemma}
\begin{proof}
 Since $\dist (X_t (x), H_i)$ is stochastically smaller than the rate one simple symmetric random walk on the set of nonnegative integers
 with a reflecting boundary at zero and starting from the same initial state, the proof of Lemma \ref{lem:stay} directly implies that
\begin{equation}
\label{eq:walk-1}
  P \,(\dist (X_t (x), H_i) > 4N \ \hbox{for some} \ t < C_3 N) \ \leq \ 3 \,\exp (- \sqrt{4N})
\end{equation}
 for all $N$ large and all $i = 1, 2, \ldots, d$.
 Moreover, by Lemmas \ref{lem:go} and \ref{lem:stay},
\begin{equation}
\label{eq:walk-2}
 \begin{array}{l}
  P \,(\dist (X_{C_3 N} (x), H_i) > N \ | \,\dist (X_t (x), H_i) \leq 4N \ \hbox{for all} \ t < C_3 N) \vspace{4pt} \\ \hspace{20pt} \leq \
  P \,(\dist (X_{T_i} (x), H_i) > N/2 \ | \,\dist (X_{T_{i - 1}} (x), H_i) \leq 4N) \vspace{4pt} \\ \hspace{40pt} + \
  P \,(\dist (X_t (x), H_i) > N \ \hbox{for some} \ t \in (T_i, C_3 N) \ | \,\dist (X_{T_i} (x), H_i) \leq N/2) \vspace{4pt} \\ \hspace{20pt} \leq \
  C_1 \,\exp (- \gamma_1 N) \ + \ 3 \,\exp (- \sqrt N) \end{array}
\end{equation}
 for all $N$ large and all $i = 1, 2, \ldots, d$.
 Finally, combining \eqref{eq:walk-1} and \eqref{eq:walk-2}, we conclude that
 $$ \begin{array}{l}
  P \,(X_t (x) \notin (- 4N, 4N]^d \ \hbox{for some} \ t < C_3 N \ \hbox{or} \ X_{C_3 N} (x) \notin (- N, N]^d) \vspace{4pt} \\ \hspace{20pt} \leq \
  P \,(X_t (x) \notin (- 4N, 4N]^d \ \hbox{for some} \ t < C_3 N) \vspace{4pt} \\ \hspace{40pt} + \
  P \,(X_{C_3 N} (x) \notin (- N, N]^d \ \hbox{and} \ X_t (x) \in (- 4N, 4N]^d \ \hbox{for all} \ t < C_3 N) \vspace{4pt} \\ \hspace{20pt} \leq \
  3d \,\exp (- \sqrt{4N}) \ + \ d \,C_1 \,\exp (- \gamma_1 N) \ + \ 3d \,\exp (- \sqrt N) \ \leq \ 7d \,\exp (- \sqrt N) \end{array} $$
 for all $N$ sufficiently large.
 This completes the proof.
\end{proof}


\subsection*{Block construction.}

\indent To complete the proof of Theorem \ref{th:invasion}, the last step is to use a rescaling argument, a technique also known as
 block construction, which has been introduced by Bramson and Durrett \cite{bramson_durrett_1988}.
 The general idea of the block construction is to couple certain good events related to the interacting particle system properly
 rescaled in space and time with the set of open sites of an oriented site percolation process on the lattice
 $$ \mathcal H \ = \ \{(z, n) = (z_1, \ldots, z_d, n) \in \Z^d \times \Z : z_1 + \ldots + z_d + n \ \hbox{is even and} \ n \geq 0 \}. $$
 We refer to Section 4 in Durrett \cite{durrett_1995} for more details about this technique and a rigorous definition of
 oriented site percolation.
 Lemmas \ref{lem:go}-\ref{lem:target} give us the appropriate space and time scales to compare the process \eqref{eq:voter} with
 oriented percolation.
 More precisely, we let $T = C_3 N$ where the constant $C_3$ is as in Lemma \ref{lem:target}, and declare $(z, n) \in \mathcal H$ to be good
 whenever the event
 $$ \Omega (z, n) \ = \ \{\xi_{nT} (x) = 2 \ \hbox{for all} \ x \in N z + (- N, N]^d \cap \Z^d \} $$
 occurs.
 Then, we have the following lemma.
\begin{lemma} --
\label{lem:invasion}
 For all $N$ sufficiently large,
 $$ \begin{array}{l}
  P \,(X_t (x) \notin (- 4N, 4N]^d \ \hbox{for some} \ (x, t) \in (- 2N, 2N]^d \times (0, T)) \vspace{4pt} \\ \hspace{50pt} + \
  P \,((\pm e_i, 1) \ \hbox{is not good for some} \ i = 1, 2, \ldots, d \ | \ (0, 0) \ \hbox{is good}) \vspace{4pt} \\ \hspace{150pt} \leq \
  7 d \,((4N)^d + 1) \,\exp (- \sqrt N) \end{array} $$
 where $e_i$ denotes the $i$th $d$-dimensional unit vector.
\end{lemma}
\begin{proof}
 By construction, there exists a dual path $(x, T) \downarrow (X_T (x), 0)$.
 Therefore, Lemma \ref{lem:target} and the duality relationship \eqref{eq:duality} imply that
 $$ \begin{array}{l}
  P \,((\pm e_i, 1) \ \hbox{is not good for some} \ i = 1, 2, \ldots, d \ | \ (0, 0) \ \hbox{is good}) \vspace{4pt} \\ \hspace{20pt} = \
  P \,(\xi_T (x) = 1 \ \hbox{for some} \ x \in (- 2N, 2N]^d \ | \ \xi_0 (y) = 2 \ \hbox{for all} \ y \in (-N, N]^d) \vspace{4pt} \\ \hspace{20pt} \leq \
  P \,(\xi_0 (X_T (x)) = 1 \ \hbox{for some} \ x \in (- 2N, 2N]^d \ | \ \xi_0 (y) = 2 \ \hbox{for all} \ y \in (-N, N]^d) \vspace{4pt} \\ \hspace{20pt} \leq \
  (4N)^d \,\sup_{\,x \in (- 2N, 2N]^d} \ P \,(\xi_0 (X_T (x)) = 1 \ | \ \xi_0 (y) = 2 \ \hbox{for all} \ y \in (-N, N]^d) \vspace{4pt} \\ \hspace{20pt} \leq \
  (4N)^d \,\sup_{\,x \in (- 2N, 2N]^d} \ P \,(X_T (x) \notin (- N, N]^d) \ \leq \ 7d \,(4N)^d \,\exp (- \sqrt N) \end{array} $$
 for all $N$ large.
 Lemma \ref{lem:target} also implies that for all $x \in (- 2N, 2N]^d$ we have
 $$ P \,(X_t (x) \notin (- 4N, 4N]^d \ \hbox{for some} \ t \in (0, T)) \ \leq \ 7d \,\exp (- \sqrt N) $$
 for all $N$ sufficiently large.
 The lemma follows.
\end{proof} \\ \\
 The right-hand side of the inequality in the statement of Lemma \ref{lem:invasion} tends to zero, therefore both probabilities on the
 left-hand side tend to zero, as $N$ tends to infinity.
 Convergence to zero of the second probability together with the translation invariance of the evolution rules of the process in space
 and time imply that, for all $\gamma > 0$, there exists $N$ large such that
\begin{equation}
\label{eq:good}
 P \,((z \pm e_i, n + 1) \ \hbox{is good for all} \ i = 1, 2, \ldots, d \ | \ (z, n) \ \hbox{is good}) \ \geq \ 1 - \gamma.
\end{equation}
 Convergence to zero of the first probability in the statement of the lemma implies that, with probability close to one when $N$ is
 large, the good event in \eqref{eq:good} only depends on the realization of the graphical representation in a finite space-time box
 which, together with \eqref{eq:good}, implies the existence of events $G_{z, n}$ such that
\begin{enumerate}
 \item $G_{z, n}$ is measurable with respect to the graphical representation in
  $$ \{N z + (- 4 N, 4 N]^d \} \times [nT, (n + 1) T] \quad \hbox{for all} \ \ (z, n) \in \mathcal H, $$
 \item For all $\gamma > 0$ there exists $N$ large such that $P \,(G_{z, n}) \geq 1 - \gamma$, \vspace{4pt}
 \item We have the inclusions of events
  $$ G_{z, n} \ \cap \ \Omega (z, n) \ \subset \ \Omega (z \pm e_i, n + 1) \quad \hbox{for all} \ \ i = 1, 2, \ldots, d. $$
\end{enumerate}
 These three conditions are the assumptions of Theorem 4.3 in Durrett \cite{durrett_1995}, from which we deduce the
 existence of a coupling such that the set of good sites contains the set of wet sites of a certain oriented site percolation
 process with parameter $1 - \gamma$ and finite range of dependency when the scale parameter $N$ is large.
 Since such a percolation process is supercritical when $\gamma$ is small enough, starting with infinitely many vertices
 of type 2, we have
 $$ \liminf_{t \to \infty} \ P \,(\xi_t (x) = 2) \ > \ 0 \quad \hbox{for all} \ x \in \Z^d. $$
 This only proves survival of the type 2 individuals, and the last step of the proof is to also establish extinction of the type 1
 individuals, which is not obvious since oriented site percolation has a positive density of closed sites.
 The existence of an in-all-directions expanding region void of type~1 is proved in the following lemma which we will apply again
 to prove Theorems \ref{th:pure-birth} and \ref{th:reduction}.
\begin{lemma} --
\label{lem:dry}
 For all $x \in \Z^d$ we have $P \,(\xi_t (x) = 1) \to 0$ as time $t \to \infty$.
\end{lemma}
\begin{proof}
 The proof follows an idea of Durrett \cite{durrett_1992} which relies on the lack of percolation of the dry sites, where a site in
 the percolation process is said to be dry if it is not wet.
 First, we define the oriented graphs $\mathcal G_1$ and $\mathcal G_2$ with common vertex set $\mathcal H$ and respective edge sets
 $$ \begin{array}{rcl}
     E_1 & = & \{((z_1, n_1), (z_2, n_2)) \in \mathcal H \times \mathcal H : z_1 - z_2 \in \{\pm e_1, \ldots, \pm e_d \} \ \hbox{and} \ n_1 = n_2 - 1 \} \vspace{4pt} \\
     E_2 & = & E_1 \,\cup \,\{((z_1, n_1), (z_2, n_2)) \in \mathcal H \times \mathcal H : z_1 - z_2 \in \{\pm 2 e_1, \ldots, \pm 2 e_d \} \ \hbox{and} \ n_1 = n_2 \}. \end{array} $$
 Given a realization of the percolation process in which sites are open with probability $1 - \gamma$, we say that there is a
 regular dry path from $(0, 0)$ to $(z, n)$ if there is a sequence
 $$ (z_0, n_0) = (0, 0), \ (z_1, n_1), \ \ldots, \ (z_k, n_k) = (z, n) \in \mathcal H $$
 such that the following two conditions hold:
\begin{enumerate}
 \item $((z_i, n_i), (z_{i + 1}, n_{i + 1})) \in E_1$ for all $i = 0, 1, \ldots, k - 1$ and \vspace{4pt}
 \item the site $(z_i, n_i)$ is dry for all $i = 0, 1, \ldots, k$.
\end{enumerate}
 Extending condition 1 by replacing $E_1$ with the larger edge set $E_2$, we call the sequence a generalized dry path.
 Assuming that initially only site $(0, 0)$ is closed, we introduce
\begin{enumerate}
 \item $D_0 (\mathcal G_1)$ = sites that can be reached from the origin by a regular dry path and \vspace{4pt}
 \item $D_0 (\mathcal G_2)$ = sites that can be reached from the origin by a generalized dry path.
\end{enumerate}
 We call these sets the regular dry cluster and the generalized dry cluster, respectively.
 Returning to the particle system, in order to prove extinction of type 1, we first modify our definition of good sites: a site
 $(z, n) \in \mathcal H$ is now said to be good if the event
 $$ \bar \Omega (z, n) \ = \ \{\xi_s (x) = 2 \ \hbox{for all} \ s \in [nT, (n + 1) T] \ \hbox{and all} \ x \in N z + (- N, N]^d \cap \Z^d \} $$
 occurs.
 The same arguments as in the proofs of Lemmas \ref{lem:go}-\ref{lem:invasion} imply that
 $$ P \,((\pm e_i, 1) \ \hbox{is not good for some} \ i = 1, 2, \ldots, d \ | \ (0, 0) \ \hbox{is good}) \ \leq \ C_4 \,N^d \,\exp (- \sqrt N) $$
 for some suitable $C_4 < \infty$, thus there again exists a coupling such that the set of good sites contains the set of wet sites of
 an oriented site percolation process with parameter $1 - \gamma$ provided the parameter $N$ is large.
 The reason for introducing this new definition of good sites is that the presence of a type 1 individual at vertex $x \in N z + (- N, N]^d$
 at time $nT$ now implies the existence of a generalized dry path from some $(z_0, 0)$ to $(z, n)$.
 In particular, the problem of extinction reduces to proving the lack of percolation of the generalized dry sites when $\gamma > 0$ is
 small enough.
 To do this, the first step is to prove the existence of $C_5 < \infty$ and $\gamma_5 > 0$ such that
\begin{equation}
\label{eq:dry-1}
  P \,(D_0 (\mathcal G_2) \not \subset ([- m, m]^d \times \R) \cap \mathcal H) \ \leq \ C_5 \,\exp (- \gamma_5 m) \quad \hbox{for $\gamma > 0$ small}.
\end{equation}
 This is essentially Theorem 5 in Durrett \cite{durrett_1992} which states the result for the regular dry cluster but his proof
 easily extends to our case.
 To begin with, we define the collection of cubes
 $$ \begin{array}{l}
      Q \ = \ \{r \in \R^{d + 1} : \sum_{i = 1, 2, \ldots, d + 1} |r_i| \leq 1 \} \quad \hbox{and} \quad
      Q (z, n) \ = \ (z, n) + Q \quad \hbox{for} \ (z, n) \in \mathcal H \end{array} $$
 and orient the faces of the cubes by assigning the value $+1$ to the $d + 1$ faces at the top of each cube, and the value $-1$ to
 the $d + 1$ faces at the bottom.
 Then, we define the contour of the regular and generalized dry clusters as the algebraic sum of all the faces of the cubes that make up
\begin{equation}
\label{eq:dry-2}
 \mathcal D_0 (\mathcal G_i) \ = \bigcup \ \{Q (z, n) : (z, n) \in D_0 (\mathcal G_i) \} \quad \hbox{for} \ i = 1, 2.
\end{equation}
 To prove \eqref{eq:dry-1}, the first ingredient is to observe that each contour has an equal number of plus and minus faces and sites below
 the plus faces must be closed.
 This holds regardless of the geometry of the underlying oriented graph so Lemma 4 in \cite{durrett_1992} directly applies:
 if the contour has length $m$ then there is a set of at least $m \,(2 (d + 1))^{-1}$ sites that must be closed.
 The second ingredient is given by Lemmas 5-6 in \cite{durrett_1992} which state that there exist $C_6 < \infty$
 and $\mu > 0$ such that the number of contours with $m$ faces is bounded by $C_6 \,\mu^m$.
 The result is proved for the regular dry cluster but again it is straightforward to extend the proof to the generalized dry cluster.
 To see this, the key is simply to observe that any cube that makes up the dry region \eqref{eq:dry-2} associated with the generalized
 dry cluster has at least one corner in common with another such cube.
 Then, following the lines of the proof of Theorem 5 in \cite{durrett_1992} and using the previous two estimates, we obtain
 $$ P \,(D_0 (\mathcal G_2) \not \subset ([- m, m]^d \times \R) \cap \mathcal H) \ \leq \ \sum_{k = m}^{\infty} \ C_6 \,\mu^k \ \gamma^{\,k \,(2 (d + 1))^{-1}} $$
 from which \eqref{eq:dry-1} follows for $\gamma < \mu^{- 2 (d + 1)}$.
 To deduce from \eqref{eq:dry-1} the lack of percolation of the dry sites, we denote by $W_n$ and $\bar W_n$ the set of wet sites at
 level $n$ in the coupled percolation processes having the same sets of open sites except at level 0 where
 $$ W_0 \ = \ \{0 \} \quad \hbox{and} \quad \bar W_0 \ = \ \{z \in \Z^d : z_1 + z_2 + \cdots + z_d \ \hbox{is even} \}. $$
 The proof of Lemma 8 in \cite{durrett_1992} gives the existence of a constant $a > 0$ such that
\begin{equation}
\label{eq:dry-3}
  \bar W_n \cap B_2 (0, an) \,\subset \,W_n \quad \hbox{for all $n$ large on the event} \ \{W_n \neq \varnothing \ \hbox{for all} \ n \geq 0 \}.
\end{equation}
 Following the proof of Lemma 11 in \cite{durrett_1992}, and using \eqref{eq:dry-1} and \eqref{eq:dry-3}, we obtain
 $$ \begin{array}{l}
    \lim_{m \to \infty} \ P \,(\hbox{a generalized dry path starting at some $(z, 0)$} \vspace{4pt} \\ \hspace{80pt}
    \hbox{intersects} \ (B_2 (0, na / 2) \times \{n \}) \cap \mathcal H \ \hbox{for some} \ n \geq m) \ = \ 0 \end{array} $$
 on the event $\{W_n \neq \varnothing \ \hbox{for all} \ n \geq 0 \}$.
 Through the coupling with the process \eqref{eq:voter}, we deduce the existence of an in-all-directions expanding region
 containing only good sites, thus the existence of an in-all-directions expanding region which is void of type 1.
\end{proof} \\ \\
 From Lemma \ref{lem:dry}, we deduce that type 2 wins for the process \eqref{eq:voter} therefore the same holds for the
 process \eqref{eq:generator-2} according to Lemma \ref{lem:coupling}, which completes the proof of Theorem \ref{th:invasion}.


\section{Selfish-selfish interactions}
\label{sec:pure-birth}

\indent This section is devoted to the proof of Theorem \ref{th:pure-birth} which claims the existence of cases of selfish-selfish
 interactions for which the most selfish type outcompetes the other type whenever starting with infinitely many representatives.
 The key to our proof is to first show that type 2 wins for the process with parameters $a_1 < 1$ and $a_2 = 1$, that we denote
 by $(\bar \eta_t)$, and then use a perturbation argument to extend the result to a neighborhood
 of this parameter point.
 First, we note that, under the assumption $a_2 = 1$, the dynamics of the process \eqref{eq:generator-2} reduce to
\begin{equation}
\label{eq:richardson}
  \begin{array}{rcl} L_{\bar \eta} f (\bar \eta) & = &
   \displaystyle \sum_{x \in \Z^d} \ \ind \,\{f_2 (x, \bar \eta) = 0 \} \ [f (\bar \eta_{x, 1}) - f (\bar \eta)] \\ && \hspace{25pt} + \
   \displaystyle \sum_{x \in \Z^d} \ \frac{(1 - a_1) \,f_2 (x, \bar \eta)}{a_1 \,f_1 (x, \bar \eta) + (1 - a_1) \,f_2 (x, \bar \eta)} \ \
                  [f (\bar \eta_{x, 2}) - f (\bar \eta)]. \end{array}
\end{equation}
 Invasion of type 2 individuals in the particular case when $a_2 = 1$ is established in the following lemma through a comparison
 with the Richardson model introduced in \cite{richardson_1973}.
\begin{lemma} --
\label{lem:richardson}
 Let $\delta > 0$.
 Then, for all $N$ sufficiently large,
 $$ P \,(\bar \eta_{C_7 N} (x) = 1 \ \hbox{for some} \ x \in (- 2N, 2N]^d \ | \ \bar \eta_0 (x) = 2 \ \hbox{for all} \ x \in (- N, N]^d) \ \leq \ \delta $$
 for some constant $C_7 < \infty$ that does not depend on $\delta$.
\end{lemma}
\begin{proof}
 First, we observe that, given that vertex $x$ has at least one neighbor of type 2 at the time it is updated, it flips from type 1
 to type 2 with probability
 $$ \frac{(1 - a_1) \,f_2 (x, \bar \eta)}{a_1 \,f_1 (x, \bar \eta) + (1 - a_1) \,f_2 (x, \bar \eta)} \ \geq \ \ep \ := \
    \frac{1 - a_1}{a_1 \,(2d - 1) + (1 - a_1)} \ > \ 0 $$
 while it flips from type 2 to type 1 with probability zero.
 This implies that, for every connected component $A \subset \Z^d$ containing at least two vertices, if all vertices in $A$ are of
 type 2 at some time, then this remains true at any later time.
 In particular, the set of type 2 individuals dominates stochastically the set of occupied sites in the Richardson model $(\xi_t)$,
 i.e., the contact process with no death, with birth parameter $\ep \,(2d)^{-1} > 0$ under certain initial conditions.
 More precisely, both processes can be constructed from the same graphical representation in such a way that
 $$ P \,(\xi_t \subset \{x : \bar \eta_t (x) = 2 \} \ \hbox{for all} \ t \geq 0 \ | \ \xi_0 \subset \{x : \bar \eta_0 (x) = 2 \}) \ = \ 1 $$
 whenever the initial set of type 2 individuals is connected and has at least two vertices.
 This, together with the Shape Theorem for the Richardson model \cite{richardson_1973}, implies that there exists a finite constant
 $C_7 < \infty$ that only depends on the parameter $\ep$ such that, for all $\delta > 0$,
 $$ \begin{array}{l}
     P \,(\bar \eta_{C_7 N} (x) = 1 \ \hbox{for some} \ x \in (- 2N, 2N]^d \ | \ \bar \eta_0 (x) = 2 \ \hbox{for all} \ x \in (- N, N]^d) \vspace{4pt} \\ \hspace{85pt} \leq \
     P \,(\xi_{C_7 N} \not \supset (- 2N, 2N]^d \ | \ \xi_0 = (- N, N]^d) \ \leq \ \delta \end{array} $$
 for all $N$ sufficiently large.
 This completes the proof.
\end{proof} \\ \\
 Extinction of type 1 individuals stated in Theorem \ref{th:pure-birth} follows from Lemmas \ref{lem:dry} and \ref{lem:richardson}
 through the comparison of the process properly rescaled in space and time with a certain oriented site percolation process in which
 dry sites do not percolate.
 Let $T = C_7 N$ where $C_7 < \infty$ is the constant fixed in Lemma \ref{lem:richardson}, and declare $(z, n) \in \mathcal H$ to be a
 good site whenever the event
 $$ \bar \Omega (z, n) \ = \ \{\bar \eta_s (x) = 2 \ \hbox{for all} \ s \in [nT, (n + 1) T] \ \hbox{and all} \ x \in N z + (- N, N]^d \cap \Z^d \} $$
 occurs.
 In the particular case when $a_1 < a_2 = 1$, individuals of type 2 with at least one neighbor of type 2 cannot change, therefore it
 follows from Lemma \ref{lem:richardson} that, for all $\delta > 0$,
 $$ P \,((z \pm e_i, n + 1) \ \hbox{is good for all} \ i = 1, 2, \ldots, d \ | \ (z, n) \ \hbox{is good}) \ \geq \ 1 - \delta $$
 for all $N$ large.
 Theorem 4.3 in Durrett \cite{durrett_1995} again gives the existence of a coupling such that the set of good sites contains the
 set of wet sites of a certain oriented site percolation process with finite range of dependency and parameter $1 - \delta$.
 To complete the proof of Theorem \ref{th:pure-birth}, we show extinction of individuals of type 1 when $a_2 < 1$ is close enough to 1
 relying on a perturbation argument.
\begin{lemma} --
\label{lem:perturbation}
 Let $a_1 < 1$.
 Then, there exists $\rho > 0$ small such that
 $$ \lim_{t \to \infty} \ P \,(\eta_t (x) = 1) \ = \ 0 \quad \hbox{for all} \ x \in \Z^d \quad \hbox{whenever} \ a_2 > 1 - \rho. $$
\end{lemma}
\begin{proof}
 First, we fix $\gamma > 0$ such that \eqref{eq:dry-1} in Lemma \ref{lem:dry} holds, and set $\delta = \gamma / 2$.
 Then, we apply Lemma \ref{lem:richardson} to obtain the existence of a scaling parameter $N$ such that
\begin{equation}
\label{eq:perturbation-1}
 \begin{array}{l}
   P \,(\bar \eta_s (x) = 1 \ \hbox{for some} \ (x, s) \in (- 2N, 2N]^d \times [T, 2T] \ | \vspace{4pt} \\ \hspace{80pt}
        \bar \eta_s (x) = 2 \ \hbox{for all} \ (x, s) \in (- N, N]^d \times [0, T]) \ \leq \ \delta.
 \end{array}
\end{equation}
 The scaling parameter being fixed, since the transition rates are continuous with respect to the parameters, there exist a
 small $\rho = \rho (N) > 0$ and a coupling of \eqref{eq:generator-2} and \eqref{eq:richardson} such that
\begin{equation}
\label{eq:perturbation-2}
\begin{array}{l}
  P \,(\eta_s (x) \neq \bar \eta_s (x) \ \hbox{for some} \ (x, s) \in (- 2N, 2N]^d \times [T, 2T] \ | \vspace{4pt} \\ \hspace{80pt}
       \eta_s (x) = \bar \eta_s (x) = 2 \ \hbox{for all} \ (x, s) \in (- N, N]^d \times [0, T]) \ \leq \ \delta
\end{array}
\end{equation}
 whenever $a_2 > 1 - \rho$.
 By combining \eqref{eq:perturbation-1} and \eqref{eq:perturbation-2}, we obtain
 $$ \begin{array}{l}
     P \,(\eta_s (x) = 1 \ \hbox{for some} \ (x, s) \in (- 2N, 2N]^d \times [T, 2T] \ | \vspace{4pt} \\ \hspace{80pt}
          \eta_s (x) = 2 \ \hbox{for all} \ (x, s) \in (- N, N]^d \times [0, T]) \ \leq \ 2 \delta \ = \ \gamma \end{array} $$
 whenever $a_2 > 1 - \rho$.
 Hence, a new application of Theorem 4.3 in Durrett \cite{durrett_1995} gives the existence of a coupling such that the set of good
 sites contains the set of wet sites of an oriented site percolation process with parameter $1 - \gamma$ when $a_2 > 1 - \rho$.
 Having such a coupling, extinction of type 1 individuals directly follows as previously from the proof of Lemma \ref{lem:dry}.
\end{proof}


\section{Strong altruistic-altruistic interactions}
\label{sec:coexistence}

\indent This section is devoted to the proof of Theorem \ref{th:coexistence} which indicates that, in contrast
 with the one-dimensional nearest neighbor model, coexistence occurs in any other spatial dimensions and for any other
 range of interactions provided cooperation is strong enough.
 The first step of the proof is to show that, excluding the one-dimensional nearest neighbor case, the critical value of the
 threshold contact process is strictly smaller than one, and then use a standard coupling argument to deduce the existence of a
 small constant $\rho > 0$ such that type 1 individuals survive for the process with parameters $a_1 = 0$ and $a_2 = \rho$.
 The coexistence result will then follow from symmetry and the monotonicity of the survival probabilities with respect to
 the parameters.
 Recall that the threshold contact process with parameter $\alpha$ is the spin system $(\xi_t)$ in which each occupied
 vertex becomes empty at rate one, and each empty vertex with at least one occupied neighbor becomes occupied at rate $\alpha$.
 Formally, the dynamics are described by the Markov generator
 $$ L_{\xi} f (\xi) \ = \ \sum_{x \in \Z^d} \ \alpha \ \ind \{f_1 (x, \xi) \neq 0 \} \ [f (\xi_{x, 1}) - f (\xi)]
                   \ + \ \sum_{x \in \Z^d} \ [f (\xi_{x, 0}) - f (\xi)] $$
 where empty and occupied vertices are identified to states zero and one, respectively.
 The reason for looking at the threshold contact process is that, when the birth rate $\alpha = 1$, it is closely
 related to the process \eqref{eq:generator-2} with parameters $a_1 = a_2 = 0$.
\begin{lemma} --
\label{lem:threshold}
 Assume that $(M, d) \neq (1, 1)$. Then,
 $$ \alpha_c \ = \ \inf \,\{\alpha : P \,(\xi_t \neq \varnothing \ \hbox{for all} \ t \geq 0 \ | \ \xi_0 = \{0 \}) > 0 \} \ < \ 1. $$
\end{lemma}
\begin{proof}
 The rate at which a vertex becomes empty is constant while the rate at which a vertex becomes occupied is a nondecreasing
 function of the configuration, where the partial order on the set of configurations is the inclusion.
 This indicates that the threshold contact process is an attractive spin system.
 In particular, Theorem 2.4 in Bezuidenhout and Gray \cite{bezuidenhout_gray_1994} implies that
\begin{equation}
\label{eq:threshold-1}
 \{\alpha : P \,(\xi_t \neq \varnothing \ \hbox{for all} \ t \geq 0 \ | \ \xi_0 = \{0 \}) > 0 \} \ \hbox{is an open interval.}
\end{equation}
 In addition, Lemma 2.1 in \cite{handjani_1999}, which relies on a result of Liggett \cite{liggett_1994}, implies that
\begin{equation}
\label{eq:threshold-2}
 P \,(\xi_t \neq \varnothing \ \hbox{for all} \ t \geq 0 \ | \ \xi_0 = \{0 \}) > 0 \quad \hbox{for} \ (M, d) \neq (1, 1) \ \hbox{and} \ \alpha \geq 1.
\end{equation}
 Combining \eqref{eq:threshold-1} and \eqref{eq:threshold-2} gives $1 \in (\alpha_c, \infty)$ therefore $\alpha_c < 1$.
\end{proof} \\ \\
 Relying on Lemma \ref{lem:threshold}, we now prove the existence of $\rho > 0$ small such that type 1 individuals survive
 for the process with parameters $a_1 = 0$ and $a_2 = \rho$, that we denote by $(\bar \eta_t)$.
 Note that, for this choice of the parameters, the dynamics reduce to
\begin{equation}
\label{eq:coexist}
  \begin{array}{rcl} L_{\bar \eta} f (\bar \eta) & = &
  \displaystyle \sum_{x \in \Z^d} \ \frac{(1 - \rho) \,f_1 (x, \bar \eta)}{(1 - \rho) \,f_1 (x, \bar \eta) + \rho \,f_2 (x, \bar \eta)} \ \
                 [f (\bar \eta_{x, 1}) - f (\bar \eta)] \vspace{4pt} \\ && \hspace{80pt} + \
  \displaystyle \sum_{x \in \Z^d} \ \ind \{f_2 (x, \bar \eta) \neq 0 \} \ [f (\bar \eta_{x, 2}) - f (\bar \eta)]. \end{array}
\end{equation}
\begin{lemma} --
\label{lem:coexist}
 For all $(M, d) \neq (1, 1)$, there is $\rho > 0$ such that
 $$ \liminf_{t \to \infty} \ P \,(\bar \eta_t (x) = 1) \ > \ 0 \quad \hbox{for all} \ x \in \Z^d. $$
\end{lemma}
\begin{proof}
 Thinking of vertices of type 1 as occupied by a particle and those of type 2 as empty, the expression of \eqref{eq:coexist}
 indicates that particles die at rate at most one, while empty vertices with at least one occupied neighbor become occupied
 at rate at least
 $$ \alpha (\rho) \ = \ \frac{1 - \rho}{1 - \rho + \rho \,((2M + 1)^d - 1)} \ = \ \frac{1 - \rho}{1 + ((2M + 1)^d - 2) \rho} \ > \ \alpha_c $$
 provided $\rho > 0$ is small enough since $\alpha_c < 1$ according to Lemma \ref{lem:threshold}.
 Therefore, for $\rho > 0$ small such that the previous inequality holds, the set of type 1 vertices dominates stochastically
 the set of occupied sites of a supercritical threshold contact process.
 Hence, it suffices to prove the result for the latter: starting with infinitely many occupied vertices,
 $$ \liminf_{t \to \infty} \ P \,(\xi_t (x) = 1) \ > \ 0 \quad \hbox{for all} \ x \in \Z^d \ \hbox{and all} \ \alpha > \alpha_c. $$
 This is a direct consequence of the complete convergence theorem for the threshold contact process established in Cox
 and Durrett \cite{cox_durrett_1991} which implies that, starting with infinitely many occupied vertices, the process converges
 in distribution to the limit obtained when starting with all vertices occupied, which is a nontrivial translation invariant
 measure when $\alpha > \alpha_c$.
\end{proof} \\ \\
 To deduce Theorem \ref{th:coexistence} from Lemma \ref{lem:coexist}, we first invoke the monotonicity of the survival
 probabilities with respect to the parameters to obtain that, for $\rho > 0$ as in the lemma, type 1 survives when the
 parameters are $a_1 = a_2 = \rho$.
 By symmetry, individuals of type 2 survive as well, and so both types coexist, for these parameters.
 Finally, invoking again the monotonicity of the survival probabilities, we deduce that coexistence occurs for all $a_1, a_2 < \rho$.


\section{Strong altruistic-weak altruistic interactions}
\label{sec:reduction}

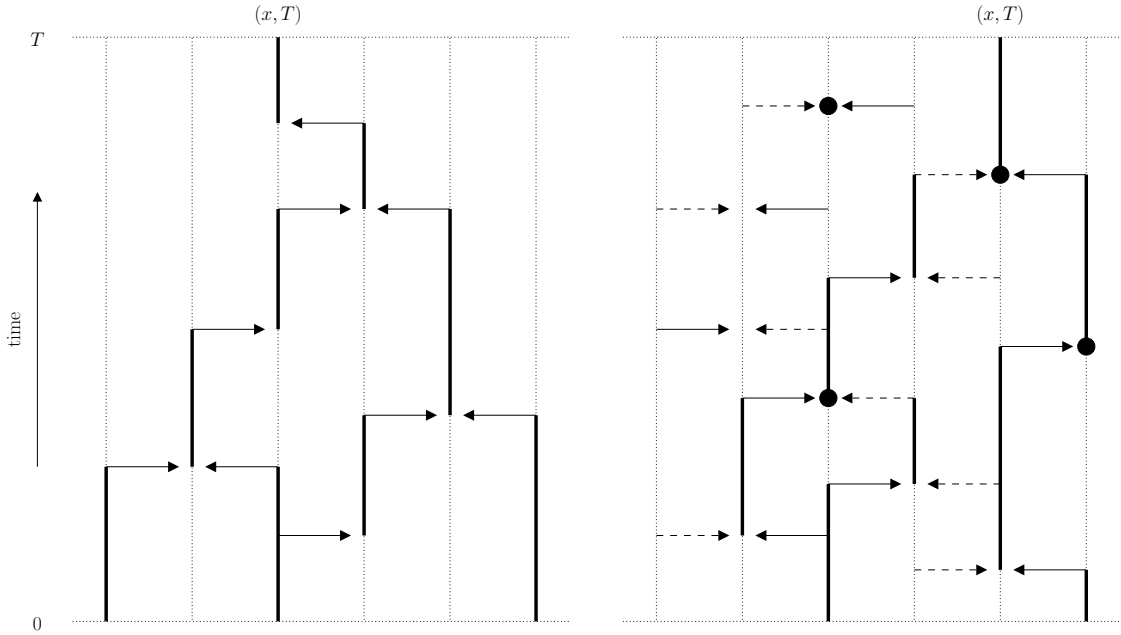
\begin{figure}[t]
\centering
\scalebox{0.36}{\input{dual.pstex_t}}
\caption{\upshape{Pictures related to the proof of Theorems \ref{th:invasion} and \ref{th:reduction}.}}
\label{fig:dual}
\end{figure}

\indent This section is devoted to the proof of Theorem \ref{th:reduction} which gives the existence of a parameter region
 in which coexistence occurs for the nonspatial model but not the spatial model.
 The proof combines ideas from Sections \ref{sec:invasion} and \ref{sec:pure-birth}.
 First, we study the process with $a_1 = 0$ and $a_2 = 1/2$ using techniques with are reminiscent of duality techniques, though
 the process does not have a dual process at this parameter point.
 This is done through the new concept of breaking point that we introduce in this paper.
 Using breaking points, we construct selected dual paths as in Section \ref{sec:invasion}, and using a block construction, we
 deduce that type 2 wins when $a_1 = 0$ and $a_2 = 1/2$.
 To complete the proof of the theorem, we rely as in Section \ref{sec:pure-birth} on a perturbation argument.


\subsection*{Breaking points.}

\indent First, we assume that $a_1 = 0$ and $a_2 = 1/2$.
 Under this assumption, the type at each vertex $x$ is updated at rate one according to the following rules.
\begin{enumerate}
 \item If vertex $x$ is of type 2 at the time of the update, then the new type is chosen uniformly at random from the set of the
  $2d$ nearest neighbors. \vspace{4pt}
 \item If vertex $x$ is of type 1 at the time of the update, then it switches to type 2 whenever at least one of its nearest
  neighbors is of type 2.
\end{enumerate}
 More formally, the dynamics of the process with $a_1 = 0$ and $a_2 = 1/2$, that we denote by $(\bar \eta_t)$, is
 described by the Markov generator
\begin{equation}
\label{eq:reduction}
  L_{\bar \eta} f (\bar \eta) \ = \
     \sum_{x \in \Z^d} \  f_1 (x, \bar \eta) \ [f (\bar \eta_{x, 1}) - f (\bar \eta)] \ + \
     \sum_{x \in \Z^d} \ \ind \{f_2 (x, \bar \eta) \neq 0 \} \ [f (\bar \eta_{x, 2}) - f (\bar \eta)].
\end{equation}
 We construct the process \eqref{eq:reduction} graphically as follows:
\begin{enumerate}
 \item for each oriented edge $e = (x, y) \in \Z^d \times \Z^d$ with $x \sim y$
 \begin{enumerate}
  \item we let $\Lambda (x, y)$ be a Poisson point process with intensity $(2d)^{-1}$,
  \item we draw at all times $t \in \Lambda (x, y)$ a solid arrow $y \to x$ and an additional $2d - 1$ dashed arrows from the
   other neighbors of $x$ to vertex $x$ to indicate that, if $x$ is of type 2 then it mimics vertex $y$ whereas if it is of type 1
   then it switches to type 2 whenever at least one of its neighbors is of type 2.
 \end{enumerate}
\end{enumerate}
\begin{defin} --
\label{def:breaking}
 We call $(x, t)$ where $t \in \Lambda (x, y)$ a breaking point if
 $$ \sup \,\{\Lambda (y, x) \cap (0, t) \} \ = \ \sup \,\{(\Lambda (x, z) \cup \Lambda (y, z)) \cap (0, t) : z \sim x \ \hbox{or} \ z \sim y \} $$
 and the set $\Lambda (y, x) \cap (0, t)$ is nonempty.
\end{defin}
 In words, $(x, t)$, $t \in \Lambda (x, y)$, is a breaking point if the next solid arrow pointing at either vertex $x$ or vertex $y$ we
 see going down the graphical representation is directed from $x$ to $y$.
 We refer to the right-hand side of Figure \ref{fig:dual} for an example of realization of the graphical representation along with
 the breaking points which are represented with black dots.
 The reason for looking at breaking points is that they allow, as in the proof of Theorem \ref{th:invasion}, to construct
 selected dual paths with a drift towards regions void of type 1 individuals and thus to prove invasion of type 2.
 This good property of the breaking points is expressed in the following lemma.
\begin{lemma} --
\label{lem:breaking}
 Assume that $(x, t)$ is a breaking point. Then
 $$ \bar \eta_t (x) = 1 \ \ \hbox{if and only if} \ \ \bar \eta_t (z) = 1 \ \hbox{for all} \ z \sim x. $$
\end{lemma}
\begin{proof}
 Assume that $t \in \Lambda (x, y)$ and let $s = \sup \,\{\Lambda (y, x) \cap (0, t) \}$. Then,
\begin{equation}
\label{eq:breaking}
 \begin{array}{rcl}
  \bar \eta_s (x) = 2 & \Longrightarrow & \bar \eta_s (y) = 2 \ \ \hbox{because there is a solid arrow $x \to y$ at time $s$} \vspace{4pt} \\
                      & \Longrightarrow & \bar \eta_r (y) = 2 \ \ \hbox{for all $r \in (s, t)$} \ \ \hbox{since no arrow points at $\{y \} \times (s, t)$} \vspace{4pt} \\
                      & \Longrightarrow & \bar \eta_t (x) = 2 \ \ \hbox{because there is a solid arrow $y \to x$ at time $t$}. \end{array}
\end{equation}
 From \eqref{eq:breaking}, we deduce that
 $$ \begin{array}{rcl}
    \bar \eta_t (x) = 1 & \Longrightarrow & \bar \eta_t (x) = 1 \ \hbox{and} \ \bar \eta_s (x) = 1 \vspace{4pt} \\
                        & \Longrightarrow & \bar \eta_t (x) = 1 \ \hbox{and} \ \bar \eta_{t-} (x) = 1 \ \ \hbox{because no arrow points at $\{x \} \times (s, t)$} \vspace{4pt} \\
                        & \Longrightarrow & \bar \eta_t (z) = 1 \ \hbox{for all} \ z \sim x \end{array} $$
 where the last implication follows from the evolution rules of the process.
 This proves the necessary condition.
 The sufficient condition is trivial and we point out that it is always true provided an arrow points at $(x, t)$ whereas the reverse
 holds because $(x, t)$ is a breaking point.
\end{proof} \\ \\
\noindent Motivated by Lemma \ref{lem:breaking}, we now construct from the graphical representation a certain system of branching coalescing
 random walks that we shall call a dual process as it allows to obtain information about the type of a given vertex, say $x$, at the
 current time, say $T$, based on the configuration at earlier times.
 To define the dual process starting at $(x, T)$, the first step is to remove from the graphical representation all the dashed arrows that
 point at a space-time point which is not a breaking point.
 Then, we say that there is a dual path from $(x, T)$ down to $(y, T - s)$, which we again write $(x, T) \downarrow (y, T - s)$, whenever
 there are sequences of times and vertices
 $$ s_0 = T - s \ < \ s_1 \ < \ \cdots \ < \ s_{m + 1} = T \qquad \hbox{and} \qquad x_0 = y, \,x_1, \,\ldots, \,x_m = x $$
 such that the following two conditions hold:
\begin{enumerate}
 \item for $i = 1, 2, \ldots, m$, there is an arrow from $x_{i - 1}$ to $x_i$ at time $s_i$ and \vspace{4pt}
 \item for $i = 0, 1, \ldots, m$, there is no arrow that points at the segments $\{x_i \} \times (s_i, s_{i + 1})$.
\end{enumerate}
 We point out that condition 1 above must hold after removal of the dashed arrows that are not directed to a breaking point.
 The dual process starting at $(x, T)$ is the set-valued process
 $$ \hat \eta_s (x, T) \ = \ \{y \in \Z^d : (x, T) \downarrow (y, T - s) \} \quad \hbox{for all} \ 0 \leq s \leq T. $$
 See the right-hand side of Figure \ref{fig:dual} where the dual process is represented in thick lines.
 One of the keys to proving Theorem \ref{th:reduction} is the following duality relationship which is simply obtained by applying
 Lemma \ref{lem:breaking} at each of the breaking points that occurs along the dual process:
\begin{equation}
\label{eq:breaking-1}
 \bar \eta_T (x) = 1 \quad \hbox{only if} \quad \bar \eta_{T - s} (y) = 1 \ \hbox{for all} \ y \in \hat \eta_s (x, T).
\end{equation}


\subsection*{Selected dual paths.}

\indent To bound the probability that space-time point $(x, T)$ is of type 1, we follow the same approach as in
 Section \ref{sec:invasion} and construct a process $(Y_t (x))$ that keeps track of a specific particle in the dual process.
 Recall that
 $$ H_i \ := \ \{z = (z_1, z_2, \ldots, z_d) \in \R^d : z_i = 0 \} \quad \hbox{for} \ i = 1, 2, \ldots, d. $$
 In addition, we let $T_i := i \,c_2 N$ for $i = 0, 1, \ldots, d$, where
 $$ c_2 \ := \ 4 \ep^{-1} \quad \hbox{with} \quad \ep \ := \ 2d \,(1 - e^{- 1/2})^2 \ e^{- (2d - 1)} \ > \ 0 $$
 and where $N$ is a scaling parameter that will be fixed later.
 The choice of $\ep$ will become clear in the proof of Lemma \ref{lem:extend} below.
 The dual path starts at $Y_0 (x) = x$.
 To construct this path at later times assuming that it is defined up to time $t$, we introduce
\begin{equation}
\label{eq:infimum}
  s (t) \ = \ \inf \,\{s > t : T - s \in \Lambda (Y_t (x), y) \ \hbox{for some} \ y \sim Y_t (x) \}.
\end{equation}
 For all $s \in (t, s (t))$, we set $Y_s (x) = Y_t (x)$ while at time $s (t)$ we have the alternative:
\begin{enumerate}
 \item If $(Y_t (x), T - s (t))$ is a breaking point and $t \in (T_{i - 1}, T_i)$ for some $i = 1, 2, \ldots, d$, then
  $$ Y_{s (t)} (x) \sim Y_t (x) \quad \hbox{and} \quad \dist (Y_{s (t)} (x), H_i) < \dist (Y_t (x), H_i). $$
 \item Otherwise, $Y_{s (t)} (x) = y$ is the neighbor for which the infimum in \eqref{eq:infimum} is reached.
\end{enumerate}
 In words, the process goes backwards in time down the graphical representation following the solid arrows from tip to tail except
 when a breaking point is encountered in which case the process crosses the arrow, either dashed or solid, that makes it closer to the
 hyperplane $H_i$ between the times $T_{i - 1}$ and $T_i$.
 After time $T_d$ only the solid arrows are crossed.
 The process is thus embedded in the dual process which, together with the relationship \eqref{eq:breaking-1}, implies that
\begin{equation}
\label{eq:breaking-2}
 \bar \eta_T (x) = 1 \quad \hbox{only if} \quad Y_T (x) = 1.
\end{equation}
 Motivated by condition \eqref{eq:breaking-2}, the next step is to prove the analog of Lemma \ref{lem:target} for this new selected dual path.
 More precisely, we have the following lemma.
\begin{lemma} --
\label{lem:extend}
 Let $x \in (- 2N, 2N]^d$.
 Then, there exists $C_8 < \infty$ such that
 $$ P \,(Y_t (x) \notin (- 4N, 4N]^d \ \hbox{for some} \ t < C_8 N \ \hbox{or} \ Y_{C_8 N} (x) \notin (- N, N]^d) \ \leq \ C_9 \,\exp (- \gamma_9 \sqrt N) $$
 for suitable $C_9 < \infty$ and $\gamma_9 > 0$, and all $N$ sufficiently large.
\end{lemma}
\begin{proof}
 To begin with, we call $s \in (0, T - 1)$ a breaking time whenever
\begin{enumerate}
 \item there is a unique neighbor $y$ of vertex $Y_s (x)$ such that there is a solid arrow $y \to Y_s (x)$ in the time interval $(T - s, T - s - 1/2)$, \vspace{4pt}
 \item there is a unique solid arrow $Y_s (x) \to y$ in the time interval $(T - s - 1/2, T - s - 1)$, \vspace{4pt}
 \item no other solid arrows point at $Y_s (x)$ or $y$ in the time interval $(T - s, T - s - 1)$.
\end{enumerate}
 Note that these three conditions imply that the dual path crosses exactly one breaking point between dual times $s$ and $s + 1$.
 Moreover, a simple calculation shows that
 $$ P \,(s \ \hbox{is a breaking time}) \ = \ 2d \,(1 - e^{- 1/2})^2 \ e^{- (2d - 1)} \ = \ \ep \ > \ 0. $$
 Since in addition disjoint parts of the graphical representations are independent, it follows that the discrete-time process
 $(Y_n (x) : n \in (0, T) \cap \Z)$ is a Markov chain that satisfies
 $$ \begin{array}{l}
     E \,(\dist (Y_{n + 1} (x), H_i) - \dist (Y_n (x), H_i)) \vspace{4pt} \\ \hspace{50pt} \leq \ \ep \
     E \,(\dist (Y_{n + 1} (x), H_i) - \dist (Y_n (x), H_i) \ | \ n \ \hbox{is a breaking time}) \ \leq \ - \ep \end{array} $$
 for all $n \in (T_{i - 1}, T_i) \cap \Z$.
 This is the analog of \eqref{eq:go-1}.
 In particular, the rest of the proof follows from the exact same arguments as in the proofs of Lemmas \ref{lem:go}-\ref{lem:target}.
\end{proof}


\subsection*{Block construction.}

\indent To complete the proof, we let $T = C_8 N$ and declare site $(z, n) \in \mathcal H$ to be a good site whenever the following
 event occurs:
 $$ \bar \Omega (z, n) \ = \ \{\bar \eta_s (x) = 2 \ \hbox{for all} \ s \in [nT, (n + 1) T] \ \hbox{and all} \ x \in N z + (- N, N]^d \cap \Z^d \}. $$
 Lemma \ref{lem:extend} and \eqref{eq:breaking-2} imply that, for all $\delta > 0$,
 $$ P \,((z \pm e_i, n + 1) \ \hbox{is good for all} \ i = 1, 2, \ldots, d \ | \ (z, n) \ \hbox{is good}) \ \geq \ 1 - \delta $$
 provided the scale parameter $N$ is sufficiently large.
 This, together with the exact same arguments as in the proofs of Lemmas \ref{lem:invasion} and \ref{lem:dry}, implies that type 2
 wins for the process \eqref{eq:reduction}.
 Finally, using the continuity of the transition rates with respect to the parameters and the same approach as in the proof
 of Lemma \ref{lem:perturbation}, we deduce the existence of $\rho > 0$ small such that type 2 wins as well for the
 process \eqref{eq:generator-2} when $a_1 < \rho$ and $a_2 > 1/2 - \rho$.
 This completes the proof of Theorem \ref{th:reduction}.




\end{document}

%% file: diagrams.pstex_t
\begin{picture}(0,0)%
\includegraphics{diagrams.pstex}%
\end{picture}%
\setlength{\unitlength}{3947sp}%
\begingroup\makeatletter\ifx\SetFigFont\undefined%
\gdef\SetFigFont#1#2#3#4#5{%
  \reset@font\fontsize{#1}{#2pt}%
  \fontfamily{#3}\fontseries{#4}\fontshape{#5}%
  \selectfont}%
\fi\endgroup%
\begin{picture}(15963,8101)(-32,-6818)
\put(8626,-5986){\makebox(0,0)[lb]{\smash{{\SetFigFont{14}{16.8}{\familydefault}{\mddefault}{\updefault}TH \ref{th:coexistence}}}}}
\put(9601,-5086){\makebox(0,0)[b]{\smash{{\SetFigFont{20}{24.0}{\rmdefault}{\bfdefault}{\updefault}coexist.}}}}
\put(8626,-3361){\makebox(0,0)[lb]{\smash{{\SetFigFont{14}{16.8}{\familydefault}{\mddefault}{\updefault}TH \ref{th:reduction}}}}}
\put(11776,-5986){\makebox(0,0)[rb]{\smash{{\SetFigFont{14}{16.8}{\familydefault}{\mddefault}{\updefault}TH \ref{th:reduction}}}}}
\put(13351,-1111){\rotatebox{45.0}{\makebox(0,0)[b]{\smash{{\SetFigFont{14}{16.8}{\familydefault}{\mddefault}{\updefault}clustering}}}}}
\put(12301,239){\makebox(0,0)[lb]{\smash{{\SetFigFont{14}{16.8}{\familydefault}{\mddefault}{\updefault}TH \ref{th:pure-birth}}}}}
\put(15001,-2461){\makebox(0,0)[b]{\smash{{\SetFigFont{14}{16.8}{\familydefault}{\mddefault}{\updefault}TH \ref{th:pure-birth}}}}}
\put(13801,-6061){\makebox(0,0)[b]{\smash{{\SetFigFont{14}{16.8}{\familydefault}{\mddefault}{\updefault}TH \ref{th:invasion}}}}}
\put(10201,239){\makebox(0,0)[b]{\smash{{\SetFigFont{14}{16.8}{\familydefault}{\mddefault}{\updefault}TH \ref{th:invasion}}}}}
\put(10201,-961){\makebox(0,0)[b]{\smash{{\SetFigFont{20}{24.0}{\rmdefault}{\bfdefault}{\updefault}type 2 wins}}}}
\put(13801,-4561){\makebox(0,0)[b]{\smash{{\SetFigFont{20}{24.0}{\rmdefault}{\bfdefault}{\updefault}type 1 wins}}}}
\put(5401,-1561){\makebox(0,0)[b]{\smash{{\SetFigFont{12}{14.4}{\familydefault}{\mddefault}{\updefault}unstable}}}}
\put(5401,-1336){\makebox(0,0)[b]{\smash{{\SetFigFont{12}{14.4}{\familydefault}{\mddefault}{\updefault}interior fixed point}}}}
\put(1801,-4936){\makebox(0,0)[b]{\smash{{\SetFigFont{12}{14.4}{\familydefault}{\mddefault}{\updefault}interior fixed point}}}}
\put(1801,-1336){\makebox(0,0)[b]{\smash{{\SetFigFont{12}{14.4}{\familydefault}{\mddefault}{\updefault}no interior}}}}
\put(1801,-1561){\makebox(0,0)[b]{\smash{{\SetFigFont{12}{14.4}{\familydefault}{\mddefault}{\updefault}fixed point}}}}
\put(5401,-4936){\makebox(0,0)[b]{\smash{{\SetFigFont{12}{14.4}{\familydefault}{\mddefault}{\updefault}no interior}}}}
\put(5401,-5161){\makebox(0,0)[b]{\smash{{\SetFigFont{12}{14.4}{\familydefault}{\mddefault}{\updefault}fixed point}}}}
\put(5401,-961){\makebox(0,0)[b]{\smash{{\SetFigFont{20}{24.0}{\rmdefault}{\bfdefault}{\updefault}bistability}}}}
\put(1801,-4561){\makebox(0,0)[b]{\smash{{\SetFigFont{20}{24.0}{\rmdefault}{\bfdefault}{\updefault}coexistence}}}}
\put(1801,1064){\makebox(0,0)[b]{\smash{{\SetFigFont{14}{16.8}{\familydefault}{\mddefault}{\updefault}$e_1$ unstable}}}}
\put(5401,1064){\makebox(0,0)[b]{\smash{{\SetFigFont{14}{16.8}{\familydefault}{\mddefault}{\updefault}$e_1$ locally stable}}}}
\put(7426,-961){\rotatebox{270.0}{\makebox(0,0)[b]{\smash{{\SetFigFont{14}{16.8}{\familydefault}{\mddefault}{\updefault}$e_2$ locally stable}}}}}
\put(7426,-4561){\rotatebox{270.0}{\makebox(0,0)[b]{\smash{{\SetFigFont{14}{16.8}{\familydefault}{\mddefault}{\updefault}$e_2$ unstable}}}}}
\put(1801,-5161){\makebox(0,0)[b]{\smash{{\SetFigFont{12}{14.4}{\familydefault}{\mddefault}{\updefault}globally stable}}}}
\put(1801,-961){\makebox(0,0)[b]{\smash{{\SetFigFont{20}{24.0}{\rmdefault}{\bfdefault}{\updefault}type 2 wins}}}}
\put(5401,-4561){\makebox(0,0)[b]{\smash{{\SetFigFont{20}{24.0}{\rmdefault}{\bfdefault}{\updefault}type 1 wins}}}}
\put(301,-6211){\makebox(0,0)[lb]{\smash{{\SetFigFont{14}{16.8}{\familydefault}{\mddefault}{\updefault}$a_2 = 0$}}}}
\put(301,-2611){\makebox(0,0)[lb]{\smash{{\SetFigFont{14}{16.8}{\familydefault}{\mddefault}{\updefault}$a_2 = 1/2$}}}}
\put(  1,-6736){\makebox(0,0)[b]{\smash{{\SetFigFont{14}{16.8}{\familydefault}{\mddefault}{\updefault}$a_1 = 0$}}}}
\put(3601,-6736){\makebox(0,0)[b]{\smash{{\SetFigFont{14}{16.8}{\familydefault}{\mddefault}{\updefault}$a_1 = 1/2$}}}}
\put(7201,-6736){\makebox(0,0)[b]{\smash{{\SetFigFont{14}{16.8}{\familydefault}{\mddefault}{\updefault}$a_1 = 1$}}}}
\put(8701,989){\makebox(0,0)[lb]{\smash{{\SetFigFont{14}{16.8}{\familydefault}{\mddefault}{\updefault}$a_2 = 1$}}}}
\put(8701,-2611){\makebox(0,0)[lb]{\smash{{\SetFigFont{14}{16.8}{\familydefault}{\mddefault}{\updefault}$a_2 = 1/2$}}}}
\put(8401,-6736){\makebox(0,0)[b]{\smash{{\SetFigFont{14}{16.8}{\familydefault}{\mddefault}{\updefault}$a_1 = 0$}}}}
\put(12001,-6736){\makebox(0,0)[b]{\smash{{\SetFigFont{14}{16.8}{\familydefault}{\mddefault}{\updefault}$a_1 = 1/2$}}}}
\put(15601,-6736){\makebox(0,0)[b]{\smash{{\SetFigFont{14}{16.8}{\familydefault}{\mddefault}{\updefault}$a_1 = 1$}}}}
\end{picture}%

%% file: dual.pstex_t
\begin{picture}(0,0)%
\includegraphics{dual.pstex}%
\end{picture}%
\setlength{\unitlength}{3947sp}%
\begingroup\makeatletter\ifx\SetFigFontNFSS\undefined%
\gdef\SetFigFontNFSS#1#2#3#4#5{%
  \reset@font\fontsize{#1}{#2pt}%
  \fontfamily{#3}\fontseries{#4}\fontshape{#5}%
  \selectfont}%
\fi\endgroup%
\begin{picture}(19452,10980)(-1739,-9526)
\put(-1199,689){\makebox(0,0)[b]{\smash{{\SetFigFontNFSS{20}{24.0}{\familydefault}{\mddefault}{\updefault}$T$}}}}
\put(-1499,-4261){\rotatebox{90.0}{\makebox(0,0)[b]{\smash{{\SetFigFontNFSS{20}{24.0}{\familydefault}{\mddefault}{\updefault}time}}}}}
\put(3001,1139){\makebox(0,0)[b]{\smash{{\SetFigFontNFSS{20}{24.0}{\familydefault}{\mddefault}{\updefault}$(x, T)$}}}}
\put(-1199,-9511){\makebox(0,0)[b]{\smash{{\SetFigFontNFSS{20}{24.0}{\familydefault}{\mddefault}{\updefault}0}}}}
\put(15601,1139){\makebox(0,0)[b]{\smash{{\SetFigFontNFSS{20}{24.0}{\familydefault}{\mddefault}{\updefault}$(x, T)$}}}}
\end{picture}%

%% file: game.bbl
\begin{thebibliography}{10}

\bibitem{athreya_swart_2005}
 Athreya, S. R. and Swart, J. M. (2005). Branching-coalescing particle systems.
\emph{Probab. Theory Related Fields} \textbf{131} 376--414.

\bibitem{bezuidenhout_gray_1994}
 Bezuidenhout, C. and Gray, L. (1994). Critical attractive spin systems.
\emph{Ann. Probab.} \textbf{22} 1160--1194.

\bibitem{bramson_durrett_1988}
 Bramson, M. and Durrett, R. (1988). A simple proof of the stability criterion of Gray and Griffeath.
\emph{Probab. Theory Related Fields} \textbf{80} 293--298.

\bibitem{bramson_griffeath_1981}
 Bramson, M. and Griffeath, D. (1981). On the Williams-Bjerknes tumour growth model. I.
\emph{Ann. Probab.} \textbf{9} 173--185.



\bibitem{cox_durrett_1991}
 Cox, J. T. and Durrett, R. (1991). Nonlinear voter models.
 In \emph{Random Walks, Brownian Motion and Interacting Particle Systems. A Festschrift in Honor of Frank Spitzer} 189--201. Birkh\"auser, Boston.


\bibitem{durrett_1992}
 Durrett, R. (1992). Multicolor particle systems with large threshold and range.
\emph{J. Theoret. Probab.} \textbf{5} 127--152.

\bibitem{durrett_1995}
 Durrett, R. (1995). Ten lectures on particle systems.
 In \emph{Lectures on probability theory (Saint-Flour, 1993)}, volume 1608 of \emph{Lecture Notes in Math.}, pages 97--201.
 Springer, Berlin.

\bibitem{durrett_levin_1994}
 Durrett, R. and Levin, S. A. (1994). The importance of being discrete (and spatial).
\emph{Theoret. Pop. Biol.} \textbf{46} 363--394.


\bibitem{handjani_1999}
 Handjani, S. J. (1999). The complete convergence theorem for coexistent threshold voter models.
\emph{Ann. Probab.} \textbf{27} 226--245.


\bibitem{kang_lanchier_2012}
 Kang, Y. and Lanchier, N. (2012). The role of space in the exploitation of resources.
\emph{Bull. Math. Biol.} \textbf{74} 1--44.

\bibitem{krone_neuhauser_1997}
 Krone, S. and Neuhauser, C. (1997). Ancestral processes with selection.
\emph{Theoret. Pop. Biol.} \textbf{51} 210--237.

\bibitem{lanchier_neuhauser_2009}
 Lanchier, N. and Neuhauser, C. (2009). Spatially explicit non Mendelian diploid model.
\emph{Ann. Appl. Probab.} \textbf{19} 1880-1920. 

\bibitem{liggett_1985}
 Liggett, T. M. (1985). \emph{Interacting particle systems},
 volume 276 of \emph{Grundlehren der Mathematischen Wissenschaften [Fundamental Principles of Mathematical Sciences]}.
 Springer-Verlag, New York.

\bibitem{liggett_1994}
 Liggett, T. M. (1994). Coexistence in threshold voter models.
\emph{Ann. Probab.} \textbf{22} 764--802.



\bibitem{neuhauser_pacala_1999}
 Neuhauser, C. and Pacala, S. W. (1999). An explicitly spatial version of the Lotka-Volterra model with interspecific competition.
\emph{Ann. Appl. Probab.} \textbf{9} 1226--1259.


\bibitem{richardson_1973}
 Richardson, D. (1973). Random growth in a tessellation.
\emph{Proc. Cambridge Philos. Soc.} \textbf{74} 515--528.

\end{thebibliography}
